\documentclass{amsart}

 \usepackage{amsmath}
\usepackage{amsfonts}
\usepackage{bbm}
\usepackage{url}
\usepackage[dvips]{graphicx}
\usepackage{color}
\usepackage{float} 
\usepackage{fullpage}
\usepackage{epsfig} 

\usepackage{mathtools}
\usepackage{mhsetup}

\usepackage{multirow}

\newtheorem{thm}{\bf{Theorem}}[section]
\newtheorem{lemma}[thm]{\bf{Lemma}}

\newtheorem{prop}[thm]{\bf{Proposition}}

\newtheorem{rem}[thm]{\bf{Remark}}

\begin{document}

\title[]{Inexact cuts in Deterministic and Stochastic Dual Dynamic Programming applied to linear optimization problems}

\maketitle

\vspace*{0.5cm}

\begin{center}
\begin{tabular}{c}
Vincent Guigues\\
School of Applied Mathematics, FGV\\
Praia de Botafogo, Rio de Janeiro, Brazil\\ 
{\tt vguigues@fgv.br}\\
\end{tabular}
\end{center}

\vspace*{0.5cm}

\begin{abstract} We introduce an extension of Dual Dynamic
Programming (DDP) to solve linear dynamic programming equations.
We call this extension IDDP-LP which 
applies to situations where some or all primal and dual subproblems to be solved along the iterations
of the method
are solved with a bounded error (inexactly). 
We provide convergence theorems both in the case when errors are bounded and
for asymptotically vanishing errors.
We extend the analysis to stochastic linear dynamic programming equations,
introducing Inexact Stochastic Dual Dynamic Programming for linear programs (ISDDP-LP), an inexact variant of SDDP 
applied to linear programs corresponding to the situation where some or all problems to be solved in the forward and backward passes of SDDP 
are solved approximately. We also provide convergence theorems for ISDDP-LP for  bounded and asymptotically vanishing errors.
Finally, we present the results of numerical experiments comparing SDDP and ISSDP-LP on a
portfolio problem with direct transaction costs modelled as a multistage stochastic linear optimization problem. On these experiments, ISDDP-LP allows us to obtain a good policy faster than SDDP.\\
\end{abstract}

\\


\section{Introduction}

Multistage stochastic convex programs are useful to model many real-life applications in engineering and finance, see for instance  
\cite{shadenrbook} and references therein. A popular solution method for such problems is Stochastic Dual Dynamic Programming (SDDP, pioneered by \cite{pereira})
which introduces sampling in the Nested Decomposition (ND) algorithm \cite{birgemulti, birge-louv-book}.
It was extended and analysed in several publications: extension for problems with interstage dependent processes \cite{morton}, \cite{guiguescoap2013},
adaptations for risk-averse problems \cite{guiguesrom12,guiguesrom10,shapsddp,kozmikmorton},
regularizations \cite{powellasamov,guilejtekregsddp}, cut selection \cite{pfeifferetalcuts,dpcuts0,guiguesbandarra17}, extension to problems with 
integer variables \cite{shabbirzou},
convergence proofs for linear programs \cite{philpot}, for nonlinear risk-neutral programs \cite{lecphilgirar12}, and for nonlinear risk-averse programs \cite{guiguessiopt2016}. Recently, 
Inexact SDDP (ISDDP) was proposed in \cite{guigues2016isddp}: it uses inexact cuts in SDDP applied to Multistage Stochastic NonLinear Problems (MSNLPs).
The motivations for ISDDP are twofold:
\begin{itemize}
\item[(i)]  first, when SDDP is applied to nonlinear problems, only approximate solutions for the subproblems solved in the backward
and forward passes are available. ISDDP allows us to build valid cuts on the basis of approximate solutions to these subproblems.
\item[(ii)] Second, for the first iterations and the  first stages, the cuts computed by SDDP can be quite
distant from the corresponding recourse function in the neighborhood of the trial point at which the cut was computed
(see for instance the numerical experiments in \cite{guiguesbandarra17, guiguesejor2017}), 
making this cut dominated by other "more relevant" cuts in this neighborhood as the method progresses.
Therefore, it is natural to try and solve less accurately, inexactly, the subproblems in the forward and backward passes
for the first iterations and stages and to increase the precision of the computed solutions as the algorithm progresses
to decrease the overall computational bulk.
\end{itemize}
The goal of this paper is to pursue this line of research considering linear instead of nonlinear programs.
More precisely, we propose and analyse a variant of SDDP applied to Multistage Stochastic Linear Programs (MSLPs)
called ISDDP-LP (Inexact SDDP for Linear Programs), which allows us to build cuts, called inexact cuts, on the basis of feasible (not
necessarily optimal and eventually far from optimal)
solutions to the subproblems solved in the forward and backward passes of the method.
The combination of inexact cuts with Benders Decomposition \cite{benders}
was first proposed by \cite{philpzakeri} for two-stage stochastic linear programs.
Therefore, ISDDP-LP can be seen as an extension to a multistage setting of the 
algorithm presented in \cite{philpzakeri}.

The main results of this paper are the following:
\begin{itemize}
\item[(A)] we propose an extension of DDP (Dual Dynamic Programming, the deterministic counterpart of SDDP) 
called IDDP-LP (Inexact DDP for Linear Programs) which builds inexact cuts for the cost-to-go functions.
For a problem with $T$ periods, when noises (error terms quantifying the inexactness) are bounded by ${\bar \delta}$
in the forward pass and by ${\bar{\varepsilon}}$ in the backward pass, we show in Theorem \ref{conviddplp}
that the limit superior of the sequence of upper bounds is at most 
$({\bar \delta} + {\bar{\varepsilon}}) \frac{T(T+1)}{2}$ distant to the optimal value of the problem
and the limit inferior of the sequence of lower bounds is at most ${\bar{\delta}} T  + {\bar{\varepsilon}} (T-1)$ distant to this optimal value.
When noises asymptotically vanish, we show that IDDP-LP solves the original optimization problem.
\item[(B)] The study of IDDP-LP allows us to introduce and analyse ISDDP-LP which builds inexact cuts for the cost-to-go functions
of a MSLP. We provide a convergence theorem (Theorem \ref{convisddplp}) for ISDDP-LP when noises are bounded
and show in Theorem \ref{convisddp1} that ISDDP-LP solves the original MSLP when noises asymptotically vanish.
\item[(C)] We compare the computational bulk of SDDP and  ISDDP-LP
on four instances of a portfolio optimization problem with direct transaction costs.
On these instances, ISDDP-LP allows us to obtain a good policy faster than SDDP (compared to SDDP, 
with ISDDP-LP the CPU time decreases
by a factor of 6.2\%, 6.4\%, 6.5\%, and 11.1\% for the four instances considered).
It  is also interesting to notice that once SDDP is implemented on a MSLP, the implementation of the corresponding ISDDP-LP
with given parameters $(\delta_t^k$, $\varepsilon_t^k)$ is straightforward.
Therefore, if for a given application, or given classes of problems, we can find suitable choices of parameters $(\delta_t^k$, $\varepsilon_t^k)$
either using the rules from Remark \ref{remchoicepssto}, other rules, or "playing" with these parameters running ISDDP-LP
on instances, ISDDP-LP could allow us to solve similar new instances quicker than SDDP.
\end{itemize}

The paper is organized as follows. In Section \ref{sec:computeinexactcuts} we explain how to build 
inexact cuts for the value function of a linear program (this elementary observation is used to build cuts
in  IDDP-LP and ISDDP-LP).
In Section \ref{iddpcon} we introduce and analyse IDDP-LP while in Section \ref{iddp} we introduce and analyse ISDDP-LP.
Numerical simulations are presented in Section \ref{numexp}.

\section{Computing inexact cuts for the value function of a linear program}\label{sec:computeinexactcuts}

Let $X \subset \mathbb{R}^m$ and let $\mathcal{Q}:X\rightarrow {\overline{\mathbb{R}}}$ be the value function given by
\begin{equation} \label{vfunctionlp}
\mathcal{Q}(x)=\left\{
\begin{array}{l}
\min_{y \in \mathbb{R}^n} \;c^T y\\
y \in Y(x):=\{y \in \mathbb{R}^n : Ay + B x = b, C y \leq f\},
\end{array}
\right.
\end{equation}
for matrices and vectors of appropriate sizes. We assume:\\
\par (H) for every $x \in X$, the set $Y(x)$ is nonempty and 
$y \rightarrow c^T y$ is bounded from below on $Y(x)$.\\
\par If Assumption (H) holds then $\mathcal{Q}$ is convex and finite on $X$ and by duality we can write
\begin{equation} \label{dualvfunctionlp}
\mathcal{Q}(x)=\left\{
\begin{array}{l}
\max_{\lambda, \mu} \;\lambda^T (b-Bx) + \mu^Tf\\
A^T \lambda + C^T  \mu = c, \mu \leq 0,
\end{array}
\right.
\end{equation}
for $x \in X$. We will call cut for $\mathcal{Q}$ on $X$ an affine lower bounding function for $\mathcal{Q}$ on $X$.
We say that cut $\mathcal{C}$ is inexact at $\bar x$ for convex function $\mathcal{Q}$
if the distance
$\mathcal{Q}( \bar x ) - \mathcal{C}( \bar x)$ between the 
values of $\mathcal{Q}$ and of the cut at $\bar x$ is strictly positive. When $\mathcal{Q}( \bar x )= \mathcal{C}( \bar x)$ we will say that cut 
$\mathcal{C}$ is exact at $\bar x$.

The following simple proposition will be used in the sequel: it provides an inexact cut for $\mathcal{Q}$ at $\bar x \in X$
on the basis of an approximate solution of \eqref{dualvfunctionlp}: 
\begin{prop} \label{inexactlp}
Let Assumption (H) hold and let $\bar x \in X$.

Let $(\hat \lambda, \hat \mu)$ be an $\epsilon$-optimal basic feasible solution for dual problem \eqref{dualvfunctionlp}
written for $x= \bar x$ (it is in particular an extreme point of the feasible set), i.e., $A^T {\hat \lambda} + C^T  {\hat \mu} = c$, $\hat \mu \leq 0$, and
\begin{equation}\label{dualsoleps}
 {\hat \lambda}^T (b-B {\bar x}) + {\hat \mu}^T f \geq \mathcal{Q}( \bar x ) - \varepsilon,
\end{equation}
for some $\varepsilon \geq 0$. Then 
the affine function
$$
\mathcal{C}(x):= {\hat \lambda}^T (b-Bx) + {\hat \mu}^T f
$$
is a cut for $\mathcal{Q}$ at $\bar x$, i.e., for every $x \in X$ we have
$\mathcal{Q}(x) \geq \mathcal{C}(x)$ and the distance 
$\mathcal{Q}( \bar x ) - \mathcal{C}( \bar x)$ between the 
values of $\mathcal{Q}$ and of the cut at $\bar x$ is at most
$\varepsilon$. 
\end{prop}
\begin{proof} $\mathcal{C}$ is indeed a cut for $\mathcal{Q}$ (an affine lower bounding function for $\mathcal{Q}$) because $(\hat \lambda, \hat \mu)$ is feasible for optimization problem \eqref{dualvfunctionlp}.
Relation \eqref{dualsoleps} gives the upper bound $\varepsilon$ for $\mathcal{Q}( \bar x ) - \mathcal{C}( \bar x)$.\hfill
\end{proof}

\section{Inexact cuts in DDP applied to linear programs} \label{iddpcon}

\subsection{Algorithm}

Consider the linear program
\begin{equation}\label{pborigdet}
\begin{array}{l}
\displaystyle \min_{x_1,\ldots,x_T \in \mathbb{R}^n} \sum_{t=1}^T c_t^T x_t \\
A_{t} x_{t} + B_{t} x_{t-1} = b_t, \;x_t \geq 0,\;t=1,\ldots,T,
\end{array}
\end{equation}
where $x_0$ is given.
For this problem we can write the following dynamic programming equations: for $t=1,\ldots,T$,
\begin{equation}\label{firststodpd}
\mathcal{Q}_t( x_{t-1} ) = \left\{
\begin{array}{l}
\displaystyle \min_{x_t \in \mathbb{R}^n} c_t^T x_t + \mathcal{Q}_{t+1} ( x_t )\\
A_{t} x_{t} + B_{t} x_{t-1} = b_t, x_t \geq 0
\end{array}
\right.
\end{equation}
with the convention that $\mathcal{Q}_{T+1}$ is null.
Clearly, the optimal value of \eqref{pborigdet} is $\mathcal{Q}_1( x_0 )$.

For convenience, we will denote 
$$
X_t(x_{t-1}):=\{x_t \in \mathbb{R}^n : A_{t} x_{t} + B_{t} x_{t-1} = b_t, \,x_t \geq 0 \}.
$$
We make the following assumption:
\begin{itemize}
\item[(H1-D)] The set $X_1(x_{0})$ is nonempty and bounded and for every $x_1 \in X_1(x_{0})$,
for every $t=2,\ldots,T$, for every $x_2 \in X_2( x_1),\ldots,x_{t-1} \in X_{t-1}( x_{t-2})$,
the set $X_t( x_{t-1} )$ is nonempty and bounded.
\end{itemize}

In this section, we introduce a variant of DDP to solve \eqref{pborigdet} 
called Inexact DDP for linear programs (IDDP-LP) where the subproblems of the forward and backward passes
are solved approximately. At iteration $k$, for $t=2,\ldots,T$,
convex function  $\mathcal{Q}_t$ is approximated by a  
 piecewise affine lower bounding function $\mathcal{Q}_t^k$ 
which is a maximum of affine lower bounding functions $\mathcal{C}_{t}^i$ called inexact cuts:
$$
\mathcal{Q}_t^k(x_{t-1} ) = \max_{1 \leq i \leq k} \mathcal{C}_{t}^i( x_{t-1}  ) \mbox{ with }\mathcal{C}_{t}^i (x_{t-1})=\theta_{t}^i + \langle \beta_{t}^i , x_{t-1} \rangle
$$
where coefficients $\theta_{t}^i, \beta_{t}^i$ are computed as explained below. The steps of IDDP-LP are as follows:\\

\par {\textbf{IDDP-LP, Step 1: Initialization.}} For $t=2,\ldots,T$, take for $\mathcal{Q}_t^0$ 
a known lower bounding affine function for $\mathcal{Q}_t$.
Set the iteration count $k$ to 1 and $\mathcal{Q}_{T+1}^0 \equiv 0$.\\
\par {\textbf{IDDP-LP, Step 2: Forward pass.}} Using approximation $\mathcal{Q}_{t+1}^{k-1}$
of $\mathcal{Q}_{t+1}$  (computed at previous iterations), we compute a $\delta_t^k$-optimal basic feasible solution $x_t^k$ of the problem (it is in particular an extreme point of the feasible set)
\begin{equation}\label{pbforwardpassd}
\left\{
\begin{array}{l}
\min_{x_t \in \mathbb{R}^n} c_t^T x_t + \mathcal{Q}_{t+1}^{k-1} ( x_t )\\
x_t \in X_t(x_{t-1}^k )
\end{array}
\right.
\end{equation}
for $t=1,\ldots,T$,
where $x_0^k=x_0$.\\
\par {\textbf{IDDP-LP, Step 3: Backward pass.}} 
The backward pass builds inexact cuts for $\mathcal{Q}_t$ at trial points $x_{t-1}^k$ computed in the forward pass.
For $k \geq 1$ and $t=1,\ldots,T$, we introduce
the function ${\underline{\mathcal{Q}}}_t^k : \mathbb{R}^n  \rightarrow {\overline{\mathbb{R}}}$ given by
\begin{equation}\label{backwardt0d}
{\underline{\mathcal{Q}}}_t^k (x_{t-1}   ) =  
\left\{
\begin{array}{l}
\min_{x_t \in \mathbb{R}^n} c_t^T x_t + \mathcal{Q}_{t+1}^k ( x_t )\\
x_t \in X_t(x_{t-1}),
\end{array}
\right.
\end{equation}
where $\mathcal{Q}_{T+1}^k \equiv 0$. We solve approximately the problem
\begin{equation}\label{backwardTd}
\mathcal{Q}_T ( x_{T-1}^k  ) = 
\left\{ 
\begin{array}{l}
\displaystyle \min_{x_T \in \mathbb{R}^n} c_{T}^T x_T \\
A_{T } x_{T} + B_{T } x_{T-1}^k = b_{T},
x_T \geq 0,
\end{array}
\right.
\mbox{ with dual }
\left\{ 
\begin{array}{l}
\max_{\lambda} \lambda^T ( b_{T} - B_{T} x_{T-1}^k )\\
A_{T}^T \lambda \leq c_{T}.
\end{array}
\right.
\end{equation}
More precisely, let $\lambda_{T}^k$ be an $\varepsilon_T^k$-optimal basic feasible solution of the dual problem above (it is in particular an extreme point of the feasible set). We 
compute 
\begin{equation}\label{thetaTkbetaTk}
\theta_{T}^k= \langle   b_{T}, \lambda_{T}^k \rangle \mbox{ and }\beta_{T}^k = - B_{T}^T \lambda_{T}^k.
\end{equation}
Using Proposition \ref{inexactlp} we have that $\mathcal{C}_{T}^k (x_{T-1})=\theta_{T}^k + \langle \beta_{T}^k , x_{T-1} \rangle$ is an inexact cut for
$\mathcal{Q}_T$ at $x_{T-1}^k$
which satisfies 
\begin{equation}\label{qualitycutTd}
\mathcal{Q}_T( x_{T-1}^k ) - \mathcal{C}_T^k ( x_{T-1}^k ) \leq \varepsilon_T^k. 
\end{equation}

Then for $t=T-1$ down to $t=2$, knowing $\mathcal{Q}_{t+1}^k \leq \mathcal{Q}_{t+1}$, consider the optimization problem
\begin{equation}\label{backwardtd}
{\underline{\mathcal{Q}}}_t^k ( x_{t-1}^k ) = 
\left\{ 
\begin{array}{l}
\displaystyle \min_{x_t} c_{t}^T x_t + \mathcal{Q}_{t+1}^k ( x_t ) \\
x_t \in X_t( x_{t-1}^k  )
\end{array}
\right.
=
\left\{ 
\begin{array}{l}
\displaystyle \min_{x_t, f} c_{t}^T x_t + f \\
A_{t} x_{t} + B_{t} x_{t-1}^k = b_{t}, x_t \geq 0,\\
f \geq \theta_{t+1}^i + \langle \beta_{t+1}^i , x_t  \rangle, i=1,\ldots,k.
\end{array}
\right.
\end{equation}
Observe that due to (H1-D) the above problem is feasible and has a finite optimal value. Therefore ${\underline{\mathcal{Q}}}_t^k ( x_{t-1}^k  )$
can be expressed as the optimal value of the corresponding dual problem:
\begin{equation}\label{dualpbtbackd}
{\underline{\mathcal{Q}}}_t^k ( x_{t-1}^k  ) = 
\left\{
\begin{array}{l}
\displaystyle \max_{\lambda, \mu} \lambda^T( b_{t} - B_{t} x_{t-1}^k   ) + \sum_{i=1}^k \mu_{i} \theta_{t+1}^i  \\
A_{t}^T \lambda +\displaystyle  \sum_{i=1}^k \mu_{i} \beta_{t+1}^i \leq c_{t},\;\sum_{i=1}^k \mu_{i}=1,\\
\mu_{i} \geq 0,\,i=1,\ldots,k.
\end{array}
\right.
\end{equation}
Let $(\lambda_{t}^k, \mu_{t}^k )$ be an  $\varepsilon_t^k$-optimal basic feasible solution of dual problem \eqref{dualpbtbackd}.
We compute 
\begin{equation}\label{formulathetabetatkd}
\theta_{t}^k =\langle  \lambda_{t}^k ,  b_{t} \rangle +  \langle    \mu_{t}^k , \theta_{t+1, k} \rangle \mbox{ and }
\beta_{t}^k =- B_{t}^T \lambda_{t}^k,
\end{equation}
where vector $\theta_{t+1, k}$ has components $\theta_{t+1}^i, i=1,\ldots,k$,
arranged in the same order as components $\mu_{t}^k(i), i=1,\ldots,k$, of $\mu_{t}^k$. 
Recalling that $\mathcal{C}_t^k ( x_{t-1} ) = \theta_t^k + \langle \beta_t^k , x_{t-1} \rangle$ and using Proposition \ref{inexactlp}, we have 
\begin{equation}\label{cuttd}
{\underline{\mathcal{Q}}}_t^k ( x_{t-1} ) \geq \mathcal{C}_t^k ( x_{t-1} ) \;\;\mbox{ and }\;\;{\underline{\mathcal{Q}}}_t^k ( x_{t-1}^k ) - \mathcal{C}_t^k  ( x_{t-1}^k )    \leq \varepsilon_t^k.
\end{equation}
Using the fact that $\mathcal{Q}_{t+1}^k( x_{t-1} )  \leq \mathcal{Q}_{t+1}( x_{t-1} )$, we have 
${\underline{\mathcal{Q}}}_t^k(x_{t-1}) \leq \mathcal{Q}_t(x_{t-1})$,
and therefore 
\begin{equation}
\mathcal{Q}_t ( x_{t-1} ) \geq \mathcal{C}_t^k ( x_{t-1} ) 
\end{equation}
which shows that $\mathcal{C}_t^k$ is a cut for $\mathcal{Q}_t$.\\
\par {\textbf{IDDP-LP, Step 4:} Do $k \leftarrow k+1$ and go to Step 2.\\

Following the proof of Lemma 1 in \cite{philpot}, we obtain that
for all $t=2,\ldots,T+1$,
the collection of distinct values $(\theta_t^k, \beta_t^k)_k$
is finite and cut coefficients $(\theta_t^k, \beta_t^k)_k$
are uniformly bounded. Observe that this proof uses the fact that 
$(\lambda_{t}^k , \mu_{t}^k)$ are extreme points of the feasible set of 
\eqref{dualpbtbackd}. There could however be unbounded sequences 
of approximate optimal feasible solutions to \eqref{dualpbtbackd}.

\subsection{Convergence analysis}

In this section we state a convergence result for IDDP-LP in Theorem \ref{conviddplp} when noises $\delta_t^k, \varepsilon_t^k$
are bounded and in Theorem \ref{convidoasa} when these noises asymptotically vanish.

We will need the following simple extension of \cite[Lemma A.1]{lecphilgirar12}:
\begin{lemma}\label{limsuptechlemma} Let $X$ be a compact set, let 
$f: X \rightarrow \mathbb{R}$ be Lipschitz continuous, and suppose that the sequence of $L$-Lipschitz continuous functions 
$f^k, k \in \mathbb{N}$ satisfies $f^{k}(x) \leq f^{k+1}(x) \leq f(x) \;\mbox{for all }x \in X,\;k \in \mathbb{N}$.
Let $(x^k)_{k \in \mathbb{N}}$ be a sequence in $X$ and  assume that 
\begin{equation}\label{limsuphyp}
\varlimsup_{k \rightarrow +\infty} f( x^k ) -f^k( x^k ) \leq S 
\end{equation}
for some $S \geq 0$.
Then
\begin{equation}\label{limsuphypkm1}
\varlimsup_{k \rightarrow +\infty} f( x^k ) -f^{k-1}( x^k ) \leq S. 
\end{equation}
\end{lemma}
\begin{proof} Let us show \eqref{limsuphypkm1} by contradiction.
Assume that  \eqref{limsuphypkm1} does not hold. Then there exist $\varepsilon_0>0$ and $\sigma: \mathbb{N} \rightarrow \mathbb{N}$
increasing such that for every $k \in \mathbb{N}$ we have
\begin{equation}\label{firstlemma}
f( x^{\sigma(k)} ) -f^{\sigma(k)-1}( x^{\sigma(k)} ) > S + \varepsilon_0.
\end{equation}
Since $x^{\sigma(k)}$ is a sequence of the compact set $X$, it has some convergent subsequence which converges to some $x_* \in X$.
Taking into account \eqref{limsuphyp} and the fact that $f^k$ are $L$-Lipschitz continuous, we can take $\sigma$
such that \eqref{firstlemma} holds and
\begin{eqnarray}
f( x^{\sigma(k)} ) - f^{\sigma(k)}( x^{\sigma(k)} ) &\leq& S + \frac{\varepsilon_0}{4},\label{seclemma}\\
f^{\sigma(k)-1}( x^{\sigma(k)} ) - f^{\sigma(k)-1}( x_* )& >& -\frac{\varepsilon_0}{4},\label{thlemma}\\
f^{\sigma(k)}( x_* ) - f^{\sigma(k)}( x^{\sigma(k)} ) &>&  -\frac{\varepsilon_0}{4}.\label{fourlemma}
\end{eqnarray}
Therefore for every $k \geq 1$ we get
$$
\begin{array}{llll}
f^{\sigma(k)}( x_* ) - f^{\sigma(k-1)}( x_* ) & \geq & f^{\sigma(k)}( x_* ) - f^{\sigma(k)-1}( x_* )&\mbox{ since }\sigma(k) \geq \sigma(k-1) + 1,\\
& = & f^{\sigma(k)}( x_* ) - f^{\sigma(k)}( x^{\sigma(k)} )&(>-\varepsilon_0/4\mbox{ by }\eqref{fourlemma}),\\
 & &  + f^{\sigma(k)}( x^{\sigma(k)} ) -   f( x^{\sigma(k)} )&(\geq-S-\varepsilon_0/4\mbox{ by }\eqref{seclemma}),\\
 && + f( x^{\sigma(k)} )  -  f^{\sigma(k)-1}( x^{\sigma(k)} )&(>S+\varepsilon_0\mbox{ by }\eqref{firstlemma}),\\
 &  &+  f^{\sigma(k)-1}( x^{\sigma(k)} )    - f^{\sigma(k)-1}( x_* )&(>-\varepsilon_0/4\mbox{ by }\eqref{thlemma}),\\
 & > & \varepsilon_0/4,&
\end{array}
$$
which implies $f^{\sigma(k)}( x_* ) > f^{\sigma(0)}( x_* ) + k \frac{\varepsilon_0}{4}$. This is in contradiction with the fact that the sequence $f^{\sigma(k)}( x_* )$ is bounded from above by $f(x_*)$.\hfill
\end{proof}

\begin{thm}[Convergence of IDDP-LP with bounded noises] \label{conviddplp}
Consider the sequences of decisions $(x_t^k)$ and of functions $(\mathcal{Q}_t^k)$ generated by IDDP-LP.
Assume that (H1-D) holds and that noises $\varepsilon_t^k$ and $\delta_t^k$ are bounded: $0 \leq \varepsilon_t^k \leq {\bar \varepsilon}$,
$0 \leq \delta_t^k \leq {\bar \delta}$ for some ${\bar \delta}, {\bar \varepsilon} \geq 0$.
\begin{itemize}
\item[(i)] Then  for $t=2,\ldots,T+1$,
\begin{equation}\label{lowerbound}
0 \leq \varliminf_{k \rightarrow +\infty} \mathcal{Q}_t( x_{t-1}^k ) - \mathcal{Q}_t^k( x_{t-1}^k ) \leq  
\varlimsup_{k \rightarrow +\infty} \mathcal{Q}_t( x_{t-1}^k ) - \mathcal{Q}_t^k( x_{t-1}^k ) \leq ({\bar \delta}  +  {\bar{\varepsilon}})(T-t+1).
\end{equation}
\item[(ii)] 
The limit superior and limit inferior of the sequence of upper bounds  $(\sum_{t=1}^T c_t^T x_t^k )_k$ on the optimal value 
$\mathcal{Q}_1( x_0 )$ of \eqref{pborigdet} satisfy
\begin{equation}\label{uppbound}
\mathcal{Q}_1( x_ 0 ) \leq 
\varliminf_{k \rightarrow +\infty}  \sum_{t=1}^T c_t^T x_t^k  \leq \varlimsup_{k \rightarrow +\infty}  \sum_{t=1}^T c_t^T x_t^k  \leq  \mathcal{Q}_1(x_0) + 
({\bar \delta} + {\bar{\varepsilon}}) \frac{T(T+1)}{2}.
\end{equation}
\item[(iii)]
The limit superior and limit inferior of the sequence of lower bounds  $({\underline{\mathcal{Q}}}_1^{k}( x_{0} ) )_k$ on the optimal value 
$\mathcal{Q}_1( x_0 )$ of \eqref{pborigdet} satisfy
\begin{equation}\label{lbound}
\mathcal{Q}_1( x_{0} )- {\bar{\delta}} T   - {\bar{\varepsilon}} (T-1) \leq \varliminf_{k \rightarrow +\infty} {\underline{\mathcal{Q}}}_1^{k}( x_{0} ) \leq  
\varlimsup_{k \rightarrow +\infty} {\underline{\mathcal{Q}}}_1^{k}( x_{0} ) \leq \mathcal{Q}_1( x_{0} ) .
\end{equation}
\end{itemize}
\end{thm}
\begin{proof} We show (i) by backward induction on $t$. Relation \eqref{lowerbound} holds for $t=T+1$.
Now assume that 
\begin{equation}\label{indhyp}
0 \leq \varliminf_{k \rightarrow +\infty} \mathcal{Q}_{t+1}( x_{t}^k ) - \mathcal{Q}_{t+1}^k( x_{t}^k ) \leq  
\varlimsup_{k \rightarrow +\infty} \mathcal{Q}_{t+1}( x_{t}^k ) - \mathcal{Q}_{t+1}^k( x_{t}^k ) \leq ( {\bar{\varepsilon}} + {\bar{\delta}} )  (T-t). 
\end{equation}
for some $t \in \{2,\ldots,T\}$.
We have
\begin{equation}\label{firstddp}
\begin{array}{lll}
0 & \leq & \mathcal{Q}_t( x_{t-1}^k ) - \mathcal{Q}_t^k( x_{t-1}^k ) \leq \mathcal{Q}_t( x_{t-1}^k ) - \mathcal{C}_t^k( x_{t-1}^k ) \mbox{ since }\mathcal{Q}_t^k \geq \mathcal{C}_t^k,\\
& \leq &  {\bar{\varepsilon}} +  \mathcal{Q}_t( x_{t-1}^k ) - {\underline{\mathcal{Q}}}_t^k( x_{t-1}^k )\mbox{ using }\eqref{cuttd}\mbox{ and }\varepsilon_t^k \leq {\bar \varepsilon},\\
& \leq &  {\bar{\varepsilon}} +  \mathcal{Q}_t( x_{t-1}^k ) - {\underline{\mathcal{Q}}}_t^{k-1}( x_{t-1}^k )\mbox{ by monotonicity},\\
& \leq &  {\bar{\varepsilon}} + {\bar{\delta}} +  \mathcal{Q}_t( x_{t-1}^k ) - c_t^T x_t^k  -  \mathcal{Q}_{t+1}^{k-1}( x_{t}^k )\mbox{ by definition of }x_t^k,\\
& = &  {\bar{\varepsilon}} + {\bar{\delta}} +  \underbrace{\mathcal{Q}_t( x_{t-1}^k ) - c_t^T x_t^k  - \mathcal{Q}_{t+1}(x_t^k)}_{\leq 0\mbox{ by definition of }\mathcal{Q}_t}+ \mathcal{Q}_{t+1}(x_t^k) -  \mathcal{Q}_{t+1}^{k-1}( x_{t}^k ),\\
& \leq & {\bar{\varepsilon}} + {\bar{\delta}} + \mathcal{Q}_{t+1}(x_t^k) -  \mathcal{Q}_{t+1}^{k-1}( x_{t}^k ).
\end{array}
\end{equation}
Using \eqref{indhyp} and applying Lemma \ref{limsuptechlemma} to $x^k=x_t^k, f^k=\mathcal{Q}_{t+1}^k, f=\mathcal{Q}_{t+1}$, we obtain 
\begin{equation}\label{secddpi}
\varlimsup_{k \rightarrow +\infty}  \mathcal{Q}_{t+1}(x_t^k) -  \mathcal{Q}_{t+1}^{k-1}( x_{t}^k ) \leq ({\bar \varepsilon} + {\bar \delta})(T-t).
\end{equation}
Combining \eqref{firstddp} and \eqref{secddpi} we obtain \eqref{lowerbound} which achieves the proof of (i).

Since $(x_1^k,x_2^k,\ldots,x_T^k)$ is feasible for \eqref{pborigdet} we have $\mathcal{Q}_1( x_0 ) \leq \sum_{t=1}^T c_t^T  x_t^k $.
Using \eqref{firstddp} we deduce
\begin{equation}\label{eqii}
\begin{array}{lll}
\mathcal{Q}_1 ( x_0 )  \leq \displaystyle \sum_{t=1}^T c_t^T x_t^k & \leq  & ({\bar \varepsilon} +{\bar \delta})T + \displaystyle \sum_{t=1}^T \mathcal{Q}_t ( x_{t-1}^k ) - \mathcal{Q}_{t+1}^{k-1}( x_t^k ) \\
& =  & ({\bar \varepsilon} +{\bar \delta})T + \displaystyle \sum_{t=1}^T \Big[ \mathcal{Q}_t ( x_{t-1}^k ) - \mathcal{Q}_{t+1}(x_t^k )\Big] + \displaystyle \sum_{t=1}^T \Big[\mathcal{Q}_{t+1}(x_t^k ) - \mathcal{Q}_{t+1}^{k-1}( x_t^k )\Big] \\
& =  & ({\bar \varepsilon} +{\bar \delta})T + \mathcal{Q}_1 ( x_0 ) + \displaystyle \sum_{t=1}^T \Big[ \mathcal{Q}_{t+1}(x_t^k ) - \mathcal{Q}_{t+1}^{k-1}( x_t^k )\Big].
\end{array}
\end{equation}
Recalling that relation \eqref{secddpi} holds for $t=1,\ldots,T$, and passing to the limit in \eqref{eqii}, we obtain (ii).

Finally,
\begin{equation}\label{finallemmalpbounded}
\begin{array}{lll}
0  \leq   \mathcal{Q}_1 ( x_0 ) - {\underline{\mathcal{Q}}}_1^k ( x_0 ) & \leq & \mathcal{Q}_1 ( x_0 ) - {\underline{\mathcal{Q}}}_1^{k-1} ( x_0 )\\
&\leq & 
{\bar \delta} + \mathcal{Q}_1 ( x_0 ) - c_1^T x_1^k - \mathcal{Q}_2^{k-1} ( x_1^k ) \\
& \leq & {\bar \delta} + \mathcal{Q}_2( x_1^k ) - \mathcal{Q}_2^{k-1} ( x_1^k ).
\end{array}
\end{equation}
Using \eqref{finallemmalpbounded} and relation \eqref{secddpi} with $t=1$, we obtain (iii).
\hfill 
\end{proof}

When all subproblems are solved exactly, i.e., when $\bar \varepsilon = \bar \delta=0$, Theorem \ref{conviddplp} 
shows that the sequences of upper bounds $\sum_{t=1}^T c_t^T x_t^k$ and of lower bounds ${\underline{\mathcal{Q}}}_1^k ( x_0 )$ converge to the optimal value of \eqref{pborigdet}
and that any accumultation point of the sequence $(x_1^k,\ldots,x_T^k)$ is an optimal solution of \eqref{pborigdet}.
Therefore, in this situation, IDDP-LP can stop when
$\sum_{t=1}^T c_t^T x_t^k - {\underline{\mathcal{Q}}}_1^k ( x_{0}  ) \leq  \mbox{Tol}$ for some parameter Tol$>0$, in which case, a Tol-optimal solution
to \eqref{pborigdet} has been found.

More generally, when noises are vanishing, i.e., when $\lim_{k \rightarrow +\infty} \varepsilon_t^k = \lim_{k \rightarrow +\infty} \delta_t^k =0$, we
can show that IDDP-LP solves \eqref{pborigdet} in a finite number of iterations:

\begin{thm}[Convergence of IDDP-LP with asymptotically vanishing noises] \label{convidoasa}
Consider the sequence of decisions $(x_1^k,\ldots$, $x_T^k)_k$ computed along the iterations of IDDP-LP.
Let Assumption (H1-D) hold. Assume that all subproblems in the forward and backward passes of IDDP-LP are solved using an algorithm
that necessarily outputs an extremal point of the feasible set, for instance the simplex algorithm.
If $\lim_{k \rightarrow +\infty} \varepsilon_t^k = \lim_{k \rightarrow +\infty} \delta_t^k =0$ for all $t$, then
there is $k_0 \in \mathbb{N}$ such that for every $k \geq k_0$, 
$(x_1^{k},x_2^{k},\ldots,x_T^{k})$
is an optimal solution of \eqref{pborigdet}.
\end{thm}
\begin{proof} Recalling that $x_1^k$ is an extremal point of $X_1(x_0)$ and that 
Assumption (H1-D) holds, IDDP-LP can only generate a finite number of different 
$x_1^k$. For each such $x_1^k$, $X_2(x_1^k)$ has a finite number of extremal points
and $x_2^k$ is one of these points. Therefore IDDP-LP can only generate a finite
number of different $x_2^k$. By induction, the number of different
trial points $x_1^k, x_2^k,\ldots, x_T^k$ is finite.
Similarly, only a finite number of different functions $\mathcal{Q}_T^k$ can be generated
(because the cut coefficients for $\mathcal{Q}_T$ are extremal points of a bounded polyhedron).
For each of these functions, a finite number of different functions $\mathcal{Q}_{T-1}^k$
can be computed.
Indeed, the number of different trial points $x_{T-1}^k$ is finite and
the  cut coefficients $\lambda_{T-1}^k$ for $\mathcal{Q}_{T-1}$ are extremal points of a bounded polyhedron.
Therefore we get a finite number of cuts ${\underline{\mathcal{Q}}}_{T-1}^{k}(x_{T-1}^k ) + \langle  \lambda_{T-1}^k , x_{T-1} - x_{T-1}^k \rangle$.
By induction, only a finite number of different functions $\mathcal{Q}_t^k, t=2,\ldots,T$, can be generated.
Therefore, after some iteration $k_1$, every optimization subproblem solved in the forward and backward passes is a copy of 
an optimization problem solved previously. It follows that after some iteration $k_0$ all subproblems are solved exactly (optimal solutions
are computed for all subproblems) and functions $\mathcal{Q}_t^k$ do not change anymore.
Consequently, from iteration $k_0$ on, to achieve the proof, we can apply the arguments of the proof of convergence of (exact) DDP (see Theorem 6.1 in \cite{guiguesejor2017}). \hfill
\end{proof}

\begin{rem}\label{remchoiceps} [Choice of parameters $\delta_t^k$ and $\varepsilon_t^k$] 
Recalling our convergence analysis and what motivates IDDP-LP, it makes sense to choose for 
$\delta_t^k$ and $\varepsilon_t^k$ sequences which decrease with $k$ and which, for fixed $k$,
decrease with $t$. A simple rule consists in defining relative errors, as long as a solver
handling such errors is used to solve the problems of the forward and backward passes.
Let the relative error for step $
t$ and iteration $k$ be ${\tt{Rel}}\_{\tt{Err}}_t^k$. 
We should take $\delta_1^k$ negligible (for instance $10^{-12}$) to compute a valid lower
bound in the first stage of the forward pass.
We propose to use  the relative error
\begin{equation}\label{rel_error}
{\tt{Rel}}\_{\tt{Err}}_t^k = \frac{1}{k}\Big[ \overline{\varepsilon} - \left(\frac{\overline{\varepsilon} - \varepsilon_0}{T-2} \right) (t-2) \Big],
\end{equation}
for step $t \geq 2$ and iteration $k \geq 1$ (in both the forward and backward passes), which induces corresponding $\delta_t^k$ and $\varepsilon_t^k$ for $t \geq 2, k \geq 1$.
With this choice, for fixed $k$, the relative error linearly decreases with $t$: it is maximal for $t=2$ (equal to some parameter $0<\overline{\varepsilon}/k<1$,
with for instance $\overline{\varepsilon}=0.1$) and minimal for $t=T$ (equal to some parameter $0<\varepsilon_0 /k<1$,
with for instance $\varepsilon_0=10^{-12}$). 

If one wishes to provide absolute errors instead of relative errors to the solvers,
we need a guess on the optimal values of the linear programs solved in the backward and forward passes. In this situation, 
we propose to
take as an estimation of
${\underline{\mathcal{Q}}}_t^k (x_{t-1}^k   )$ the value ${\underline{\mathcal{Q}}}_t^{k-1} (x_{t-1}^k   )$
which is available before solving \eqref{dualpbtbackd} since it was computed in the  forward pass of iteration $k$.
We take all $\delta_t^k$ negligible and define the absolute errors
\begin{equation} \label{abserrorddp}
\varepsilon_t^k = \max\Big(1, \left|{\underline{\mathcal{Q}}}_t^{k-1} (x_{t-1}^k ) \right|   \Big) {\tt{Rel}}\_{\tt{Err}}_t^k
\end{equation}
where ${\tt{Rel}}\_{\tt{Err}}_t^k$ is given by \eqref{rel_error}.
\end{rem}

\section{Inexact cuts in SDDP applied to linear programs} \label{iddp}

\subsection{Problem formulation assumptions, and algorithm}

We are interested in solution methods for linear Stochastic Dynamic Programming equations:
the first stage problem is 
\begin{equation}\label{firststodp}
\mathcal{Q}_1( x_0 ) = \left\{
\begin{array}{l}
\min_{x_1 \in \mathbb{R}^n} c_1^T x_1 + \mathcal{Q}_2 ( x_1 )\\
A_{1} x_{1} + B_{1} x_{0} = b_1,
x_1 \geq 0
\end{array}
\right.
\end{equation}
for $x_0$ given and for $t=2,\ldots,T$, $\mathcal{Q}_t( x_{t-1} )= \mathbb{E}_{\xi_t}[ \mathfrak{Q}_t ( x_{t-1}, \xi_{t}  )  ]$ with
\begin{equation}\label{secondstodp} 
\mathfrak{Q}_t ( x_{t-1}, \xi_{t}  ) = 
\left\{ 
\begin{array}{l}
\min_{x_t \in \mathbb{R}^n} c_t^T x_t + \mathcal{Q}_{t+1} ( x_t )\\
A_{t} x_{t} + B_{t} x_{t-1} = b_t,
x_t \geq 0,
\end{array}
\right.
\end{equation}
with the convention that $\mathcal{Q}_{T+1}$ is null and
where for $t=2,\ldots,T$, random vector $\xi_t$ corresponds to the concatenation of the elements in random matrices $A_t, B_t$ which have a known
finite number of rows and random vectors $b_t, c_t$.
Moreover, it is assumed that $\xi_1$ is not random. For convenience, we will denote 
$$
X_t(x_{t-1}, \xi_t):=\{x_t \in \mathbb{R}^n : A_{t} x_{t} + B_{t} x_{t-1} = b_t, \,x_t \geq 0 \}.
$$
We make the following assumptions:
\begin{itemize}
\item[(H1-S)] The random vectors $\xi_2, \ldots, \xi_T$ are independent and have discrete distributions with finite support. 
\item[(H2-S)] The set $X_1(x_{0}, \xi_1 )$ is nonempty and bounded and for every $x_1 \in X_1(x_{0}, \xi_1 )$,
for every $t=2,\ldots,T$, for every realization $\tilde \xi_2, \ldots, \tilde\xi_t$ of $\xi_2,\ldots,\xi_t$,
for every $x_{\tau} \in X_{\tau}( x_{\tau-1} , \tilde \xi_{\tau}), \tau=2,\ldots,t-1$, the set $X_t( x_{t-1} , {\tilde \xi}_t )$
is nonempty and bounded.
\end{itemize}
We put $\Theta_1 = \{\xi_1\}$ and for $t\geq 2$
we will denote by $\Theta_t = \{\xi_{t 1},\ldots,\xi_{t M_t} \}$ the support  of $\xi_t$ for stage $t$ with
$p_{t i}= \mathbb{P}(\xi_t = \xi_{t i}) >0, i=1,\ldots,M_t$ and with vector $\xi_{t j}$ being the concatenation
of the elements in $A_{t j}, B_{t j}, b_{t j}, c_{t j}$.

Inexact SDDP applied to linear Stochastic Dynamic Programming equations \eqref{firststodp}, \eqref{secondstodp}
is a simple extension of SDDP, called ISDDP-LP, where the  subproblems of the forward and backward passes
are solved approximately. At iteration $k$, for $t=2,\ldots,T$,
 function  $\mathcal{Q}_t$ is approximated by a  
 piecewise affine lower bounding function $\mathcal{Q}_t^k$ 
which is a maximum of affine lower bounding functions $\mathcal{C}_{t}^i$ called inexact cuts:
$$
\mathcal{Q}_t^k(x_{t-1} ) = \max_{1 \leq i \leq k} \mathcal{C}_{t}^i( x_{t-1}  ) \mbox{ with }\mathcal{C}_{t}^i (x_{t-1})=\theta_{t}^i + \langle \beta_{t}^i , x_{t-1} \rangle
$$
where coefficients $\theta_{t}^i, \beta_{t}^i$ are computed as explained below. The steps of ISDDP-LP are as follows.\\

\par {\textbf{ISDDP-LP, Step 1: Initialization.}} For $t=2,\ldots,T$, take for $\mathcal{C}_t^0 =\mathcal{Q}_t^0$ 
a known lower bounding affine function for $\mathcal{Q}_t$. Set the iteration count $k$ to 1 and $\mathcal{Q}_{T+1}^0 \equiv 0$.\\
\par {\textbf{ISDDP-LP, Step 2: Forward pass.}} We generate a sample
${\tilde \xi}^k = (\tilde \xi_1^k, \tilde \xi_2^k,\ldots,\tilde \xi_T^k)$ from the distribution of $\xi^k \sim (\xi_1,\xi_2,\ldots,\xi_T)$,
with the convention that $\tilde \xi_1^k = \xi_1$. Using approximation $\mathcal{Q}_{t+1}^{k-1}$
of $\mathcal{Q}_{t+1}$  (computed at previous iterations), we compute a $\delta_t^k$-optimal basic feasible solution $x_t^k$ of the problem
\begin{equation}\label{pbforwardpass}
\left\{
\begin{array}{l}
\min_{x_t \in \mathbb{R}^n} x_t^T {\tilde c}_t^k + \mathcal{Q}_{t+1}^{k-1} ( x_t )\\
x_t \in X_t(x_{t-1}^k, {\tilde \xi}_{t}^k )
\end{array}
\right.
\end{equation}
for $t=1,\ldots,T$,
where $x_0^k=x_0$ and where $\tilde c_t^k$ is the realization of $c_t$ in $\tilde \xi_t^k$. For $k \geq 1$ and $t=1,\ldots,T$, define the function ${\underline{\mathfrak{Q}}}_t^k : \mathbb{R}^n {\small{\times}} \Theta_t \rightarrow  {\overline{\mathbb{R}}}$ by
\begin{equation}\label{backwardt0}
{\underline{\mathfrak{Q}}}_t^k (x_{t-1} , \xi_t   ) =  
\left\{
\begin{array}{l}
\min_{x_t \in \mathbb{R}^n} c_t^T x_t + \mathcal{Q}_{t+1}^k ( x_t )\\
x_t \in X_t(x_{t-1}, \xi_{t} ).
\end{array}
\right.
\end{equation}
With this notation, we have 
\begin{equation}\label{forwdefQ}
{\underline{\mathfrak{Q}}}_t^{k-1} (x_{t-1}^k , {\tilde \xi}_t^k   )
\leq  \langle {\tilde c}_t^k ,  x_t^k \rangle + \mathcal{Q}_{t+1}^{k-1}( x_t^k ) \leq {\underline{\mathfrak{Q}}}_t^{k-1} (x_{t-1}^k , {\tilde \xi}_t^k   ) + \delta_t^k.
\end{equation}

\par {\textbf{ISDDP-LP, Step 3: Backward pass.}} 
The backward pass builds inexact cuts for $\mathcal{Q}_t$ at $x_{t-1}^k$ computed in the forward pass.
For $t=T+1$,  we have $\mathcal{Q}_{t}^k = \mathcal{Q}_{T+1}^k \equiv 0$, i.e., $\theta_{T+1}^k$ and $\beta_{T+1}^k$ are null. 
For $j=1,\ldots,M_T$, we solve approximately the problem
\begin{equation}\label{backwardT}
\mathfrak{Q}_T ( x_{T-1}^k, \xi_{T j}  ) = 
\left\{ 
\begin{array}{l}
\displaystyle \min_{x_T \in \mathbb{R}^n} c_{T j}^T x_T \\
A_{T j} x_{T} + B_{T j} x_{T-1}^k = b_{T j},
x_T \geq 0,
\end{array}
\right.
\mbox{ with dual }
\left\{ 
\begin{array}{l}
\max_{\lambda} \lambda^T ( b_{T j} - B_{T j} x_{T-1}^k )\\
A_{T j}^T \lambda \leq c_{T j}.
\end{array}
\right.
\end{equation}
Let $\lambda_{T j}^k$ be an $\varepsilon_T^k$-optimal basic feasible solution of the dual problem above:
$A_{T j}^T \lambda_{T j}^k \leq c_{T j}$ and 
\begin{equation}\label{epssolbackT}
\mathfrak{Q}_T ( x_{T-1}^k, \xi_{T j}  ) - \varepsilon_T^k \leq   \langle \lambda_{T j}^k , b_{T j} - B_{T j} x_{T-1}^k \rangle  \leq \mathfrak{Q}_T ( x_{T-1}^k, \xi_{T j}  ). 
\end{equation}
We 
compute 
\begin{equation}\label{thetaTkbetaTk}
\theta_{T}^k= \sum_{j=1}^{M_T} p_{T j} \langle   b_{T j}, \lambda_{T j}^k \rangle \mbox{ and }\beta_{T}^k = -\sum_{j=1}^{M_T} p_{T j} B_{T j}^T \lambda_{T j}^k.
\end{equation}
Using Proposition \ref{inexactlp} we have that $\mathcal{C}_{T}^k (x_{T-1})=\theta_{T}^k + \langle \beta_{T}^k , x_{T-1} \rangle$ is an inexact cut for
$\mathcal{Q}_T$ at $x_{T-1}^k$. Using \eqref{epssolbackT}, we also see that 
\begin{equation}\label{qualitycutT}
\mathcal{Q}_T( x_{T-1}^k ) - \mathcal{C}_T^k ( x_{T-1}^k ) \leq \varepsilon_T^k. 
\end{equation}

Then for $t=T-1$ down to $t=2$, knowing $\mathcal{Q}_{t+1}^k \leq \mathcal{Q}_{t+1}$,
for $j=1,\ldots,M_t$, consider the optimization problem
\begin{equation}\label{backwardt}
{\underline{\mathfrak{Q}}}_t^k ( x_{t-1}^k, \xi_{t j}  ) = 
\left\{ 
\begin{array}{l}
\displaystyle \min_{x_t} c_{t j}^T x_t + \mathcal{Q}_{t+1}^k ( x_t ) \\
x_t \in X_t( x_{t-1}^k , \xi_{t j} )
\end{array}
\right.
=
\left\{ 
\begin{array}{l}
\displaystyle \min_{x_t, f} c_{t j}^T x_t + f \\
A_{t j} x_{t} + B_{t j} x_{t-1}^k = b_{t j}, x_t \geq 0,\\
f \geq \theta_{t+1}^i + \langle \beta_{t+1}^i , x_t  \rangle, i=1,\ldots,k,
\end{array}
\right.
\end{equation}
with optimal value ${\underline{\mathfrak{Q}}}_t^k ( x_{t-1}^k, \xi_{t j}  )$.
Observe that due to (H2-S) the above problem is feasible and has a finite optimal value. Therefore ${\underline{\mathfrak{Q}}}_t^k ( x_{t-1}^k, \xi_{t j}  )$
can be expressed as the optimal value of the corresponding dual problem:
\begin{equation}\label{dualpbtback}
{\underline{\mathfrak{Q}}}_t^k ( x_{t-1}^k, \xi_{t j}  ) = 
\left\{
\begin{array}{l}
\displaystyle \max_{\lambda, \mu} \lambda^T( b_{t j} - B_{t j} x_{t-1}^k   ) + \sum_{i=1}^k \mu_{i} \theta_{t+1}^i  \\
A_{t j}^T \lambda +\displaystyle  \sum_{i=1}^k \mu_{i} \beta_{t+1}^i \leq c_{t j},\;\sum_{i=1}^k \mu_{i}=1,\\
\mu_{i} \geq 0,\,i=1,\ldots,k.
\end{array}
\right.
\end{equation}
Let $(\lambda_{t j}^k, \mu_{t j}^k )$ be an  $\varepsilon_t^k$-optimal basic  feasible solution of dual problem \eqref{dualpbtback}
and let ${\underline{\mathcal{Q}}}_t^k$ be the function given by ${\underline{\mathcal{Q}}}_t^k ( x_{t-1} ) = \sum_{j=1}^{M_t} p_{t j} {\underline{\mathfrak{Q}}}_t^k( x_{t-1} , \xi_{t j} )$.
We compute 
\begin{equation}\label{formulathetabetatk}
\theta_{t}^k =\sum_{j=1}^{M_t} p_{t j}\Big(  \langle  \lambda_{t j}^k ,  b_{t j} \rangle +  \langle    \mu_{t j}^k , \theta_{t+1, k} \rangle \Big) \mbox{ and }
\beta_{t}^k =- \sum_{j=1}^{M_t} p_{t j}B_{t j}^T \lambda_{t j}^k,
\end{equation}
where vector $\theta_{t+1, k}$ has components $\theta_{t+1}^i, i=1,\ldots,k$,
arranged in the same order as components $\mu_{t j}^k(i), i=1,\ldots,k$, of $\mu_{t j}^k$. Setting $\mathcal{C}_t^k ( x_{t-1} ) = \theta_t^k + \langle \beta_t^k , x_{t-1} \rangle$ and using Proposition \ref{inexactlp}, we have 
\begin{equation}\label{cutt}
{\underline{\mathcal{Q}}}_t^k ( x_{t-1} ) \geq \mathcal{C}_t^k ( x_{t-1} ) \;\;\mbox{ and }\;\;{\underline{\mathcal{Q}}}_t^k ( x_{t-1}^k ) - \mathcal{C}_t^k  ( x_{t-1}^k )    \leq \varepsilon_t^k.
\end{equation}
Using the fact that $\mathcal{Q}_{t+1}^k( x_{t-1} )  \leq \mathcal{Q}_{t+1}( x_{t-1} )$, we have 
${\underline{\mathfrak{Q}}}_t^k(x_{t-1}, \xi_{t j}) \leq \mathfrak{Q}_t(x_{t-1}, \xi_{t j})$,
${\underline{\mathcal{Q}}}_t^k(x_{t-1}) \leq \mathcal{Q}_t(x_{t-1})$,
and therefore 
\begin{equation}
\mathcal{Q}_t ( x_{t-1} ) \geq \mathcal{C}_t^k ( x_{t-1} ) 
\end{equation}
which shows that $\mathcal{C}_t^k$ is an inexact cut for $\mathcal{Q}_t$.\\
\par {\textbf{ISDDP-LP, Step 4:} Do $k \leftarrow k+1$ and go to Step 2.\\

\if{
\par The previous algorithm can be modified solving in the backward pass some of the subproblems only. 
This gives rise to Inexact DOASA whose steps are given below in the case when the randomness only appears
in the right-hand side of the constraints, namely when $\xi_t=b_t$.\\

\par Steps 1, 2, and 4 of IDOASA are identical to Steps 1, 2, and 4 of ISDDP-LP.\\
\par  {\textbf{Step 4 of IDOASA is as follows}}: set $\mathcal{D}_T^k = \emptyset$,
take a sample $\Theta_T^k \subset \Theta_T$
and for each $b_{T j} \in \Theta_t^k$ compute an $\varepsilon_T^k$-optimal solution $\lambda_{T j}^k$
of dual problem \eqref{backwardT} and add $\lambda_{T j}^k$ to set $\mathcal{D}_T^k$. For each $b_{T j} \notin \Theta_t^k$, compute 
$$
\lambda_{T j}^k = \mbox{argmax} \{ \lambda^T(b_{T j} - B_{T j} x_{T-1}^k ) \;  : \; \lambda \in \mathcal{D}_T^k\}.
$$
Compute $\theta_T^k, \beta_{T}^k$ by \eqref{thetaTkbetaTk}.
For $t=T-1$ down to $t=2$, set $\mathcal{D}_t^k = \emptyset$, take a sample $\Theta_t^k \subset \Theta_t$
and for each $b_{t j} \in \Theta_t^k$ compute an $\varepsilon_t^k$-optimal solution $(\lambda_{t j}^k, \mu_{t j}^k)$
of dual problem \eqref{dualpbtback} and add $(\lambda_{t j}^k, \mu_{t j}^k )$ to set $\mathcal{D}_t^k$. For each $b_{t j} \notin \Theta_t^k$, compute 
$$
(\lambda_{t j}^k , \mu_{t j}^k )  = \mbox{argmax} \{\lambda^T( b_{t j} - B_{t j} x_{t-1}^k   ) + \sum_{i=1}^k \mu_{i} \theta_{t+1}^i  \;  : \; (\lambda, \mu) \in \mathcal{D}_t^k \}.
$$
Compute $\theta_t^k, \beta_{t}^k$ by \eqref{formulathetabetatk}.\\

\par To put it in a nutshell, IDOASA is a variant of DOASA where all subproblems solved in the backward and 
forward passes are solved approximately, namely $\varepsilon_t^k$-optimal solutions are computed for the subproblems
of stage $t$ and iteration $k$.
Observe that since $\xi_t=b_t$, we have $A_{t j}=A_t$, $c_{t j}=c_T$ and therefore
for fixed $k$ dual problems \eqref{backwardT} (when $b_{T j}$ varies in $\Theta_T^k$)
have the same feasible sets. Similarly, for fixed $k$ dual problems \eqref{dualpbtback} 
(when $b_{t j}$ varies in $\Theta_t^k$) have the same feasible sets. This justifies that 
$\mathcal{C}_t^k$ computed by IDOASA is a valid cut (an affine lower bounding function) for $\mathcal{Q}_t$.
}\fi

Similarly to IDDP-LP, the collection of distinct values $(\theta_t^k, \beta_t^k)_k$
is finite and cut coefficients $(\theta_t^k, \beta_t^k)_k$
are uniformly bounded. 

\subsection{Convergence analysis}\label{conv-analsddp}

In this section we state a convergence result for ISDDP-LP in Theorem \ref{convisddplp} when noises $\delta_t^k, \varepsilon_t^k$
are bounded and in Theorem \ref{convisddp1} when these noises asymptotically vanish.

We will assume that the sampling procedure in ISDDP-LP satisfies the following property:\\

\par (H3-S) The samples in the backward passes are independent: $(\tilde \xi_2^k, \ldots, \tilde \xi_T^k)$ is a realization of
$\xi^k=(\xi_2^k, \ldots, \xi_T^k) \sim (\xi_2, \ldots,\xi_T)$ 
and $\xi^1, \xi^2,\ldots,$ are independent.\\

Before stating our first convergence theorem, we need more notation. Due to Assumption (H1-S), the realizations of $(\xi_t)_{t=1}^T$ form a scenario tree of depth $T+1$
where the root node $n_0$ associated to a stage $0$ (with decision $x_0$ taken at that
node) has one child node $n_1$
associated to the first stage (with $\xi_1$ deterministic).
We denote by $\mathcal{N}$ the set of nodes and for a node $n$ of the tree, we define: 
\begin{itemize}
\item $C(n)$: the set of children nodes (the empty set for the leaves);
\item $x_n$: a decision taken at that node;
\item $p_n$: the transition probability from the parent node of $n$ to $n$;
\item $\xi_n$: the realization of process $(\xi_t)$ at node $n$\footnote{The same notation $\xi_{\tt{Index}}$ is used to denote
the realization of the process at node {\tt{Index}} of the scenario tree and the value of the process $(\xi_t)$
for stage {\tt{Index}}. The context will allow us to know which concept is being referred to.
In particular, letters $n$ and $m$ will only be used to refer to nodes while $t$ will be used to refer to stages.}:
for a node $n$ of stage $t$, this realization $\xi_n$ contains in particular the realizations
$c_n$ of $c_t$, $b_n$ of $b_t$, $A_{n}$ of $A_{t}$, and $B_{n}$ of $B_{t}$.
\end{itemize}

Next, we define for iteration $k$ decisions $x_n^k$ for all node $n$ of the scenario tree
simulating the policy obtained in the end of iteration $k-1$ replacing 
cost-to-go function $\mathcal{Q}_t$ by 
$\mathcal{Q}_{t}^{k-1}$ for $t=2,\ldots,T+1$:
\rule{\linewidth}{1pt}
{\textbf{Simulation of the policy in the end of iteration $k-1$.}}\\
\hspace*{0.4cm}{\textbf{For }}$t=1,\ldots,T$,\\
\hspace*{0.8cm}{\textbf{For }}every node $n$ of stage $t-1$,\\
\hspace*{1.6cm}{\textbf{For }}every child node $m$ of node $n$, compute a $\delta_t^k$-optimal solution $x_m^k$ of
\begin{equation} \label{defxtkj}
{\underline{\mathfrak{Q}}}_t^{k-1}( x_n^k , \xi_m ) = \left\{
\begin{array}{l}
\displaystyle \inf_{x_m} \; c_m^T x_m  + \mathcal{Q}_{t+1}^{k-1}( x_m ) \\
x_m \in X_t( x_n^k, \xi_m ),
\end{array}
\right.
\end{equation}
\hspace*{2.4cm}where $x_{n_0}^k = x_0$.\\
\hspace*{1.6cm}{\textbf{End For}}\\
\hspace*{0.8cm}{\textbf{End For}}\\
\hspace*{0.5cm}{\textbf{End For}}\\
\vspace*{-0.6cm}\\
\rule{\linewidth}{1pt}\\
We are now in a position to state our first convergence theorem for ISDDP-LP:
\begin{thm}[Convergence of ISDDP-LP with bounded noises] \label{convisddplp}
Consider the sequences of decisions $(x_n^k)_{n \in \mathcal{N}}$ and of functions $(\mathcal{Q}_t^k)$ generated by ISDDP-LP.
Assume that (H1-S), (H2-S), (H3-S) hold, and that 
noises $\varepsilon_t^k$ and $\delta_t^k$ are bounded: $0 \leq \varepsilon_t^k \leq {\bar \varepsilon}$,
$0 \leq \delta_t^k \leq {\bar \delta}$ for finite ${\bar \delta}, {\bar \varepsilon}$.
Then the following holds:
\begin{itemize}
\item[(i)] for $t=2,\ldots,T+1$, for all node $n$ of stage $t-1$,  almost surely
\begin{equation}\label{lowerbounds}
0 \leq \varliminf_{k \rightarrow +\infty} \mathcal{Q}_t( x_{n}^k ) - \mathcal{Q}_t^k( x_{n}^k ) \leq  
\varlimsup_{k \rightarrow +\infty} \mathcal{Q}_t( x_{n}^k ) - \mathcal{Q}_t^k( x_{n}^k ) \leq ({\bar \delta}  +  {\bar{\varepsilon}})(T-t+1);
\end{equation}
\item[(ii)] for every $t=2,\ldots,T$, for all node $n$ of stage $t-1$,
the limit superior and limit inferior of the sequence of upper bounds  $\Big( \displaystyle  \sum_{m \in C(n)} p_m (  c_m^T x_m^k   + \mathcal{Q}_{t+1}( x_m^k ))  \Big)_k$ satisfy almost surely
\begin{equation}\label{uppbounds}
\begin{array}{l}
0 \leq 
\varliminf_{k \rightarrow +\infty} \displaystyle \sum_{m \in C(n)} p_m \Big[ c_m^T x_m^k   + \mathcal{Q}_{t+1}( x_m^k ) \Big] - \mathcal{Q}_t( x_n^k ),  \\
\varlimsup_{k \rightarrow +\infty} \displaystyle  \sum_{m \in C(n)} p_m \Big[ c_m^T x_m^k   + \mathcal{Q}_{t+1}( x_m^k ) \Big] - \mathcal{Q}_t( x_n^k ) \leq ({\bar \delta} + {\bar{\varepsilon}})(T-t+1);
\end{array}
\end{equation}
\item[(iii)]
the limit superior and limit inferior of the sequence of lower bounds  $({\underline{\mathfrak{Q}}}_1^{k-1}( x_{0} , \xi_1 ) )_k$ on the optimal value 
$\mathcal{Q}_1( x_0 )$ of \eqref{firststodp} satisfy almost surely
\begin{equation}\label{lbounds}
\mathcal{Q}_1( x_{0} )- {\bar{\delta}} T   - {\bar{\varepsilon}} (T-1) \leq \varliminf_{k \rightarrow +\infty} {\underline{\mathfrak{Q}}}_1^{k-1}( x_{0} , \xi_1 ) \leq  
\varlimsup_{k \rightarrow +\infty} {\underline{\mathfrak{Q}}}_1^{k-1}( x_{0} , \xi_1 ) \leq \mathcal{Q}_1( x_{0} ).
\end{equation}
\end{itemize}
\end{thm}
\begin{proof} The proof is a simple combination of arguments from the proof of 
Theorem 3.1 in \cite{lecphilgirar12} and of Theorem \ref{conviddplp} from Section \ref{iddpcon}. For interested readers, the detailed proof is provided in the appendix.\hfill
\end{proof}

Theorem \ref{convisddp1} below shows the convergence of ISDDP-LP in a finite number of iterations when noises $\varepsilon_t^k, \delta_t^k$ asymptotically vanish.

\begin{thm} [Convergence of ISDDP-LP with asymptotically vanishing noises] \label{convisddp1}
Consider the sequences of decisions $(x_n^k)$ and of functions $(\mathcal{Q}_t^k)$ generated by ISDDP-LP.
Let Assumptions (H1-S), (H2-S), and (H3-S) hold.
If $\lim_{k \rightarrow +\infty} \delta_t^k = \lim_{k \rightarrow +\infty} \varepsilon_t^k = 0$,
then ISDDP-LP converges with probability one in a finite number of iterations to an optimal solution to
\eqref{firststodp}, \eqref{secondstodp}.
\end{thm}
\begin{proof}
The arguments are similar to the proof of Theorem \ref{convidoasa}.
Due to Assumptions (H1-S), (H2-S), ISDDP-LP 
generates almost surely a finite number of 
trial points $x_1^k, x_2^k,\ldots, x_T^k$.
Similarly, almost surely only a finite number of different functions $\mathcal{Q}_t^k, t=2,\ldots,T,$ can be generated.
Therefore, after some iteration $k_1$, every optimization subproblem solved in the forward and backward passes is a copy of 
an optimization problem solved previously. It follows that after some iteration $k_0$ all subproblems are solved exactly (optimal solutions
are computed for all subproblems) and functions $\mathcal{Q}_t^k$ do not change anymore.
Consequently, from iteration $k_0$ on, we can apply the arguments of the proof of convergence of (exact) SDDP applied to linear programs (see Theorem 5 in \cite{philpot}). \hfill 
\end{proof}

\begin{rem}\label{remchoicepssto} [Choice of parameters $\delta_t^k$ and $\varepsilon_t^k$]
As for IDDP-LP, we take $\delta_1^k$ is negligible (for instance $10^{-12}$) and the relative error
\begin{equation}\label{rel_errors2}
{\tt{Rel}}\_{\tt{Err}}_t^k = \frac{1}{k}\Big[ \overline{\varepsilon} - \left(\frac{\overline{\varepsilon} - \varepsilon_0}{T-2} \right) (t-2) \Big],
\end{equation}
for step $t \geq 2$ and iteration $k \geq 1$ (in both the forward and backward passes), which induces corresponding $\delta_t^k$ and $\varepsilon_t^k$ for $t \geq 2, k \geq 1$.

However, it seems more delicate to define sound absolute errors. Ultimately, absolute errors \eqref{abserrorddp}
used for IDDP-LP could be replaced by
\begin{equation} \label{abserrorddpsto}
\varepsilon_t^k = \max\Big(1, \left|{\underline{\mathfrak{Q}}}_t^{k-1} (x_{t-1}^k ,  \tilde \xi_t^k ) \right|   \Big) {\tt{Rel}}\_{\tt{Err}}_t^k
\end{equation}
with ${\tt{Rel}}\_{\tt{Err}}_t^k$ still given by \eqref{rel_errors2}.
\end{rem}

\section{Numerical experiments} \label{numexp}

Our goal in this section is to compare SDDP and ISDDP-LP (denoted for short ISDDP in what follows) on the risk-neutral portfolio problem with direct transaction costs presented in
Section 5.1 of \cite{guilejtekregsddp} (see \cite{guilejtekregsddp} for details).
For this application, $\xi_t$ is the vector of asset returns: if 
$n$ is the number of
risky assets, $\xi_t$ has size $n+1$,
$\xi_t(1:n)$ is the vector of risky asset returns for stage $t$ 
while $\xi_t(n+1)$ is the return of the risk-free asset. We generate four instances of this portfolio problem as follows.

For fixed $T$ (number of stages) and $n$ (number of risky assets),
the distributions of $\xi_t(1:n),t=2,\ldots,T$,  have $M$ realizations 
with $p_{t i}=\mathbb{P}(\xi_t = \xi_{t i})=1/M$, and
$\xi_1(1:n), \xi_{t 1}(1:n), \ldots, \xi_{t M}(1:n)$ obtained sampling
from a normal distribution with mean and standard deviation chosen 
randomly in respectively the intervals $[0.9,1.4]$ and $[0.1,0.2]$.
The monthly return $\xi_t(n+1)$ of the risk-free asset is $1.01$ for all $t$.
The initial portfolio $x_0$ has components uniformly distributed in $[0,10]$ (vector of initial wealth in each asset).
The largest possible position in any security is set to $u_i=20\%$.  
Transaction costs are known with
$\nu_t(i)=\mu_t(i)$
obtained sampling from the distribution of the random variable 
$0.08+0.06\cos(\frac{2\pi}{T} U_T )$ where $U_T$ is a random variable
with a discrete distribution over the set of integers $\{1,2,\ldots,T\}$.
Our four instances of the portfolio problem are obtained taking
for $(M,T,n)$ the combinations of values $(100,10,50)$, $(100,30,50)$, $(50,20,50)$, and 
$(50,40,10)$.
All linear subproblems of the forward and backward passes
are solved numerically using Mosek solver \cite{mosek} and
for ISDDP, we solve approximately these subproblems limiting the number of iterations 
of Mosek solver as indicated in Table \ref{tablenumberiter0} in the Appendix.
The strategy given in this table is (as indicated in Remark \ref{remchoicepssto}) to increase the accuracy
(or, equivalently, increase the maximal number of iterations allowed for Mosek solver)
of the solutions to subproblems as ISDDP iteration increases and 
for a given iteration of ISDDP, to increase the accuracy 
(or, equivalently, increase the maximal number of iterations allowed for Mosek solver)
of the solutions to subproblems as the number of stages increases from $t=2$ to $t=T$, knowing that we
solve exactly the subproblems for the last stage $T$ and for the first stage $t=1$.

SDDP and ISDDP were implemented in Matlab and the code was run on a Xeon E5-2670 processor with 
384 GB of RAM. 
For a given instance, SDDP and ISDDP were run using the same set of sampled scenarios 
along iterations. We stopped  SDDP algorithm when the gap is $<10\%$
and run ISDDP for the same number of iterations.\footnote{The gap is defined as $\frac{Ub-Lb}{Ub}$ where $Ub$ and $Lb$ correspond to upper and lower bounds, respectively. 
Though the portfolio problem is a maximization 
problem (of the mean income), we have rewritten it as a minimization problem (of the mean loss), of form \eqref{firststodp}, \eqref{secondstodp}.
The lower bound $Lb$ is the optimal value of the first stage problem and the 
upper  bound $Ub$ is the upper end of a 97.5\%-one-sided confidence interval on the optimal value for $N=100$ policy realizations, see 
\cite{shapsddp} for a detailed discussion on this stopping criterion.}

On our four instances, we then simulate the policies obtained with SDDP and ISDDP
on a set of 500 scenarios of returns. The gap between the two policies on these scenarios and the CPU time reduction
using ISDDP are given in 
Table \ref{tablecpuisddp}. In this table, the gap is defined by 
$100 \frac{\tt{Cost ISDDP}-\tt{Cost SDDP}}{\tt{Cost SDDP}}$
where 
${\tt{Cost ISDDP}}$ and $\tt{Cost SDDP}$ are respectively the mean cost for ISDDP and
SDDP policies on the 500 simulated scenarios and the CPU time reduction is given by
$100 \frac{\tt{Time SDDP}-\tt{Time ISDDP}}{\tt{Time SDDP}}$
where $\tt{Time SDDP}$ and $\tt{Time ISDDP}$ correspond to the time needed to 
compute SDDP and ISDDP policies (before running the Monte Carlo simulation), respectively. 

On all instances the gap is relatively small and ISDDP policy is computed faster than SDDP policy.

\begin{table}[H]
\centering
\begin{tabular}{|c|c|c|c|c|}
\hline
$M$ & $T$ & $n$ & Gap (\%) & CPU time reduction (\%) \\
\hline
50 & 20 &   50  & 0.1 & 6.2 \\
\hline
50 & 40 &   10  & 4.2 & 11.1 \\
\hline 
100 & 10 &   50  & 0.8 & 6.5 \\
\hline 
100 & 30 &   50  & 3.4 & 6.4 \\
\hline 
\end{tabular}
\caption{Empirical gap between SDDP and ISDDP policies and CPU time reduction for ISDDP over SDDP.}\label{tablecpuisddp}
\end{table}

\par More precisely, we report in Figure \ref{fig1sddp} (for instances with $(M,T,n)=(100,10,50)$ and $(M,T,n)=(100,30,50)$)
and Figure \ref{fig2sddp} (for instances with $(M,T,n)=(50,20,50)$
and $(M,T,n)=(50,40,10)$) three outputs along the iterations
of SDDP and ISDDP: the cumulative CPU time (in seconds), the number of iterations needed for Mosek LP solver to solve
all backward and forward subproblems, and the upper and lower bounds on the optimal value computed
by the methods (note that the upper bounds are only computed from iteration 100 on, because the past
$N=100$ iterations are used to compute them). 

\par These experiments (i) show that it is possible to 
obtain a near optimal policy quicker than SDDP solving approximately some subproblems in SDDP and (ii) confirm that ISDDP computes a valid lower bound since
first stage subproblems are solved exactly. For the first iterations, this lower bound can however be
distant from SDDP lower bound (see for instance the bottom left plots of Figures 
\ref{fig1sddp} and \ref{fig2sddp}). However, both  SDDP and ISDDP lower and upper bounds are 
quite close after 200 iterations, even if Mosek LP solver uses much less iterations to solve
the subproblems with ISDDP (see the middle plots of Figures \ref{fig1sddp}, \ref{fig2sddp}). 
The total CPU time needed by ISDDP is significantly inferior but this CPU time reduction decreases
when the number of iterations increases. If many iterations are required to solve the problem,
after a few hundreds iterations backward and forward subproblems are solved in similar CPU time for SDDP and ISDDP and the total CPU time
reduction starts to stabilize.

\begin{figure}
\centering
\begin{tabular}{cc}
\includegraphics[scale=0.6]{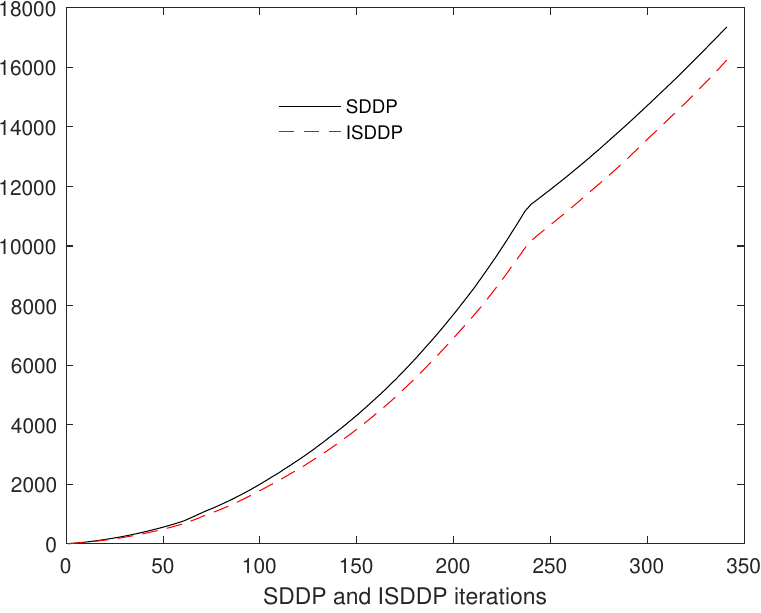}
&
\includegraphics[scale=0.6]{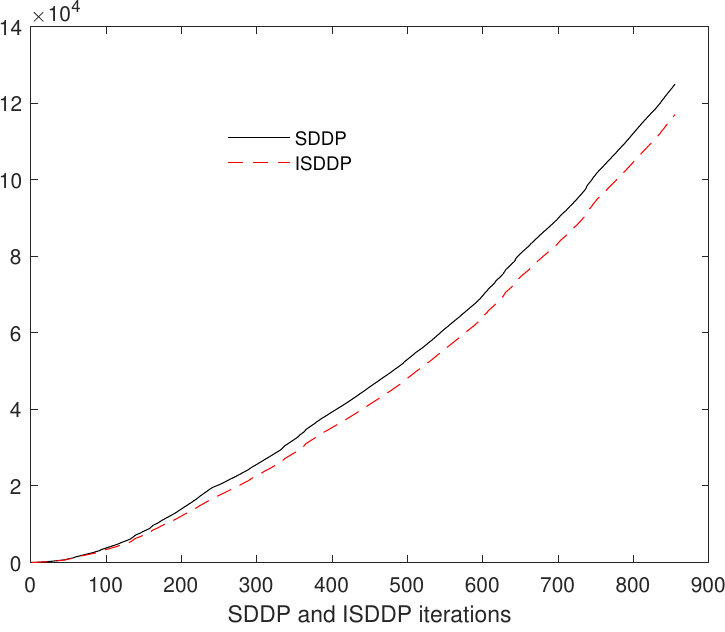}
\\
\includegraphics[scale=0.6]{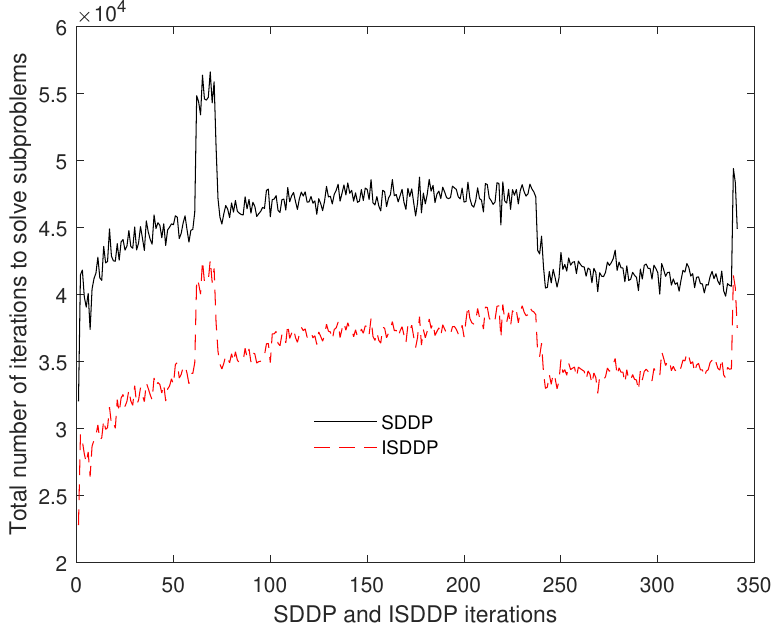}
&
\includegraphics[scale=0.6]{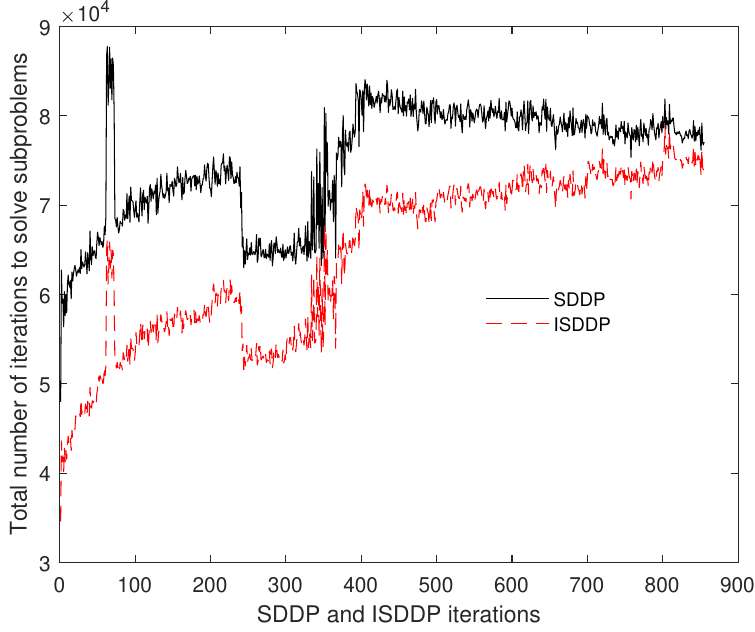}
\\
\includegraphics[scale=0.6]{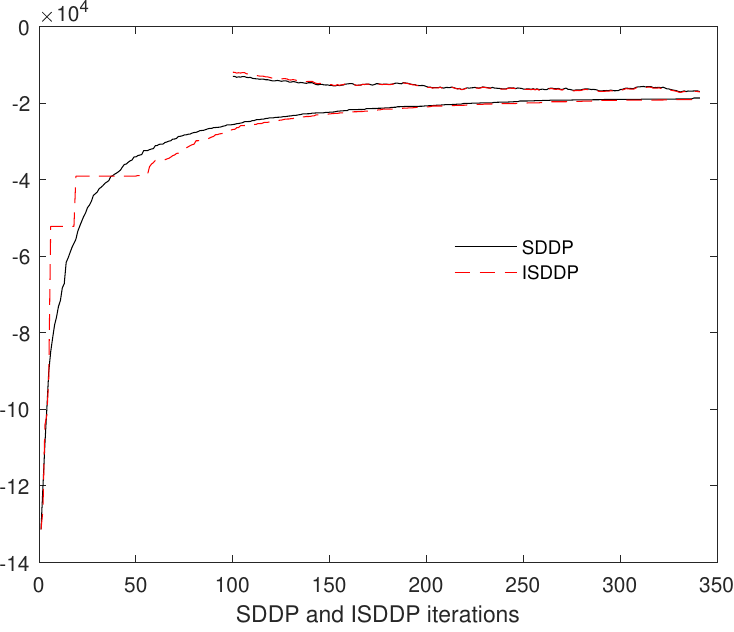}
&
\includegraphics[scale=0.6]{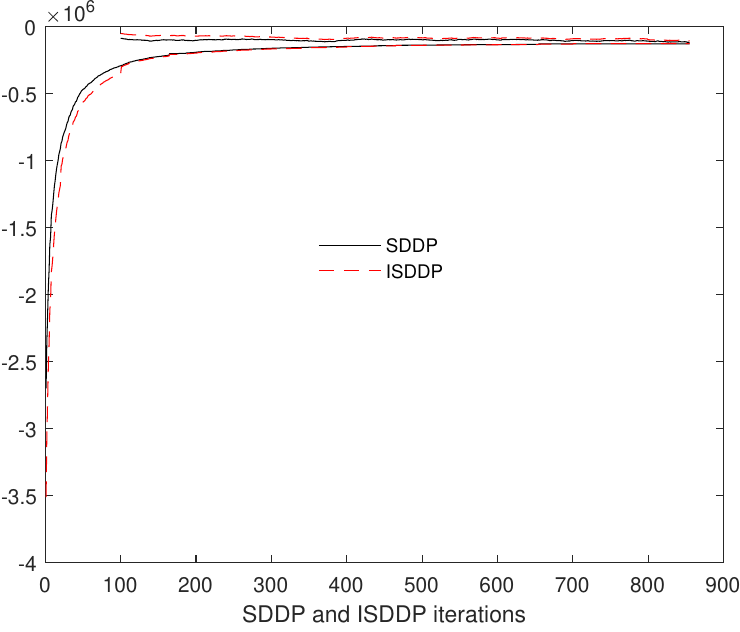}
\end{tabular}
\caption{Top plots: cumulative CPU time (in seconds), middle plots: total number of iterations to solve subproblems, bottom 
plots: upper and lower bounds. Left plots: $M=100,$ $T=10$, $n=50$, right plots: $M=100,$ $T=30$, and $n=50$.}
\label{fig1sddp}
\end{figure}

\begin{figure}
\centering
\begin{tabular}{cc}
\includegraphics[scale=0.6]{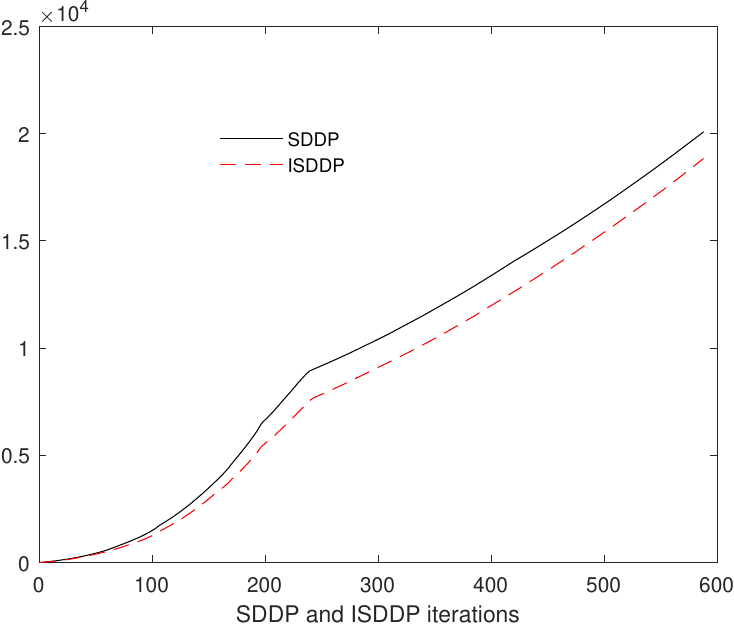}
&
\includegraphics[scale=0.6]{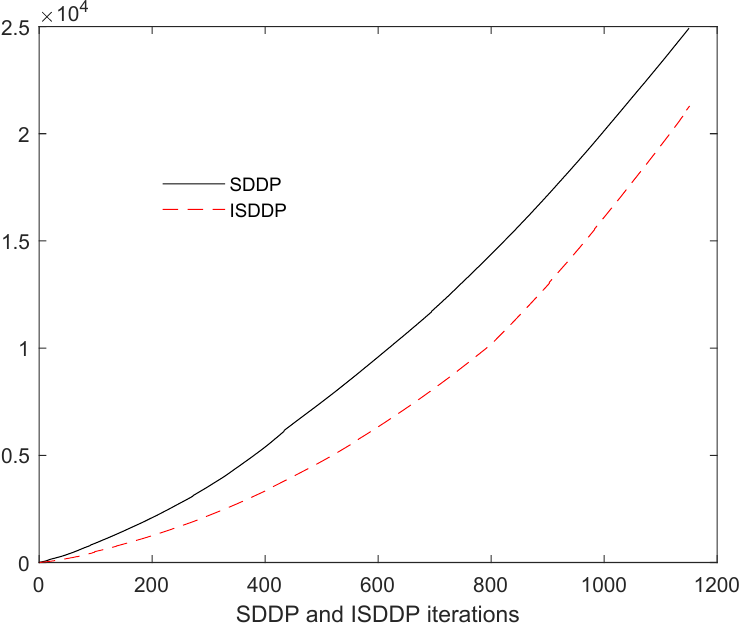}
\\
\includegraphics[scale=0.6]{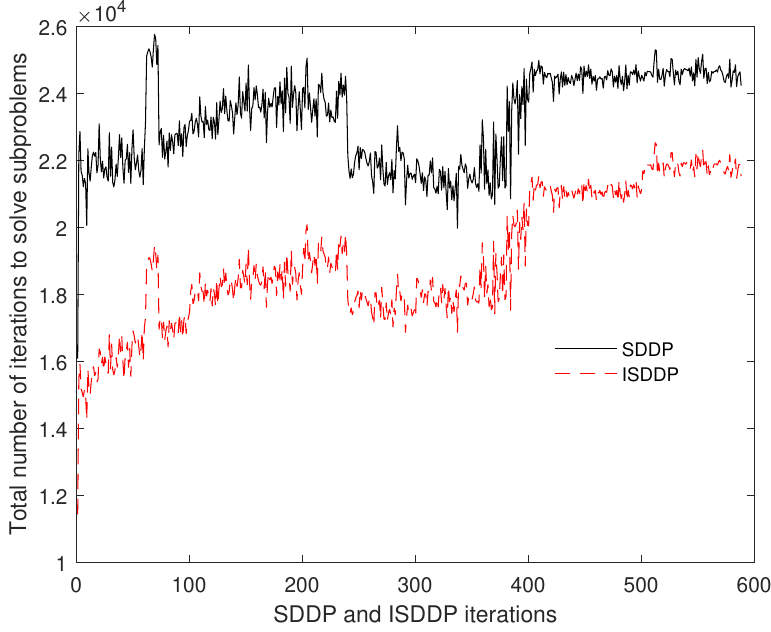}
&
\includegraphics[scale=0.6]{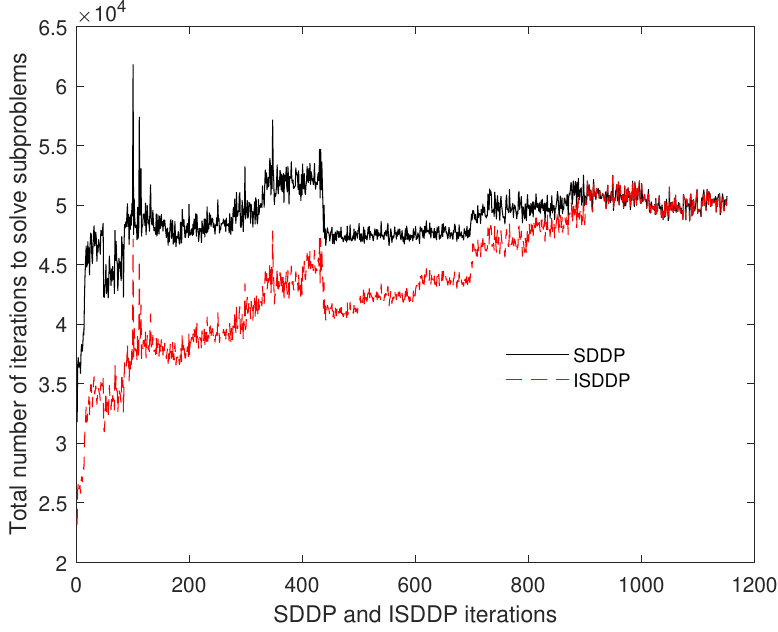}
\\
\includegraphics[scale=0.6]{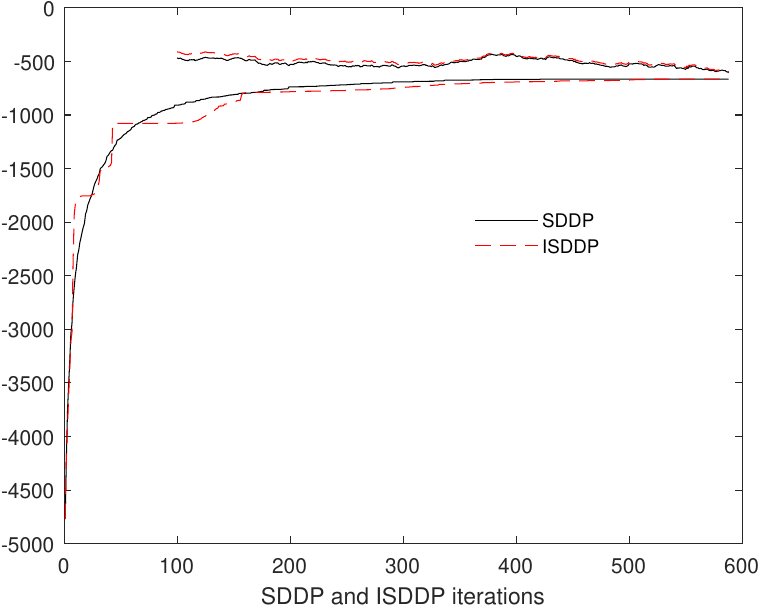}
&
\includegraphics[scale=0.6]{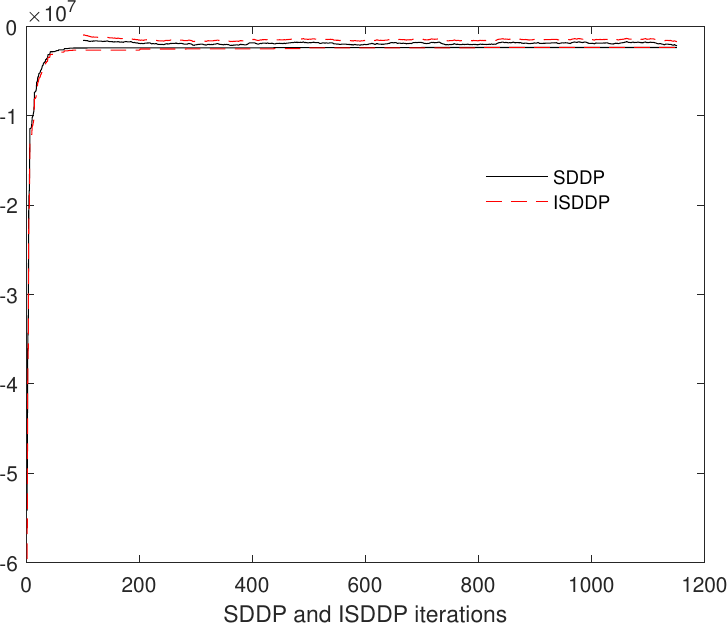}
\end{tabular}
\caption{Top plots: cumulative CPU time (in seconds), middle plots: total number of iterations to solve subproblems, bottom 
plots: upper and lower bounds. Left plots: $M=50,$ $T=20$, $n=50$, right plots: $M=50,$ $T=40$, and $n=10$.}\label{fig2sddp}
\end{figure}

\section{Conclusion}

We introduced IDDP-LP and ISDDP-LP, the first inexact variants of DPP and SDDP applied to 
respectively linear programs and
multistage stochastic linear programs. We studied the convergence 
of IDDP-LP and ISDDP-LP and presented the results of numerical experiments comparing the computational bulk 
of SDDP and ISDDP-LP on a portfolio problem.

Since ISDDP-LP can be much quicker than SDDP for some well chosen parameters $(\delta_t^k, \varepsilon_t^k)$
and is straightforward to implement from SDDP, it would be interesting to use ISDDP-LP
on other real-life applications modelled by multistage stochastic linear programs.

As a continuation of this work, it would also be interesting to consider a variant of SDDP
that builds cuts in the backward pass on the basis of approximate solutions which are not 
necessarily feasible (but of course asymptotically feasible to derive a convergence result).

\section*{Acknowledgments} The author's research was 
partially supported by an FGV grant, CNPq grants 311289/2016-9 and 
401371/2014-0, and FAPERJ grant E-26/201.599/2014. 

\nocite{*}

\addcontentsline{toc}{section}{References}
\bibliographystyle{plain}
\bibliography{Inexact_SDDP_MSLP}

\section*{Appendix}

\par {\textbf{Proof of Theorem \ref{convisddplp}.}}\\

\par (i) We show \eqref{lowerbounds} for $t=2,\ldots,T+1$, and all node $n$ of stage $t-1$ by backward induction on $t$. The relation
holds for $t=T+1$. Now assume that it holds for $t+1$ for some $t \in \{2,\ldots,T\}$. Let us show that it holds for $t$.
Take a node $n$ of stage $t-1$.
Observe that the sequence $\mathcal{Q}_t( x_n^k ) - \mathcal{Q}_t^k ( x_n^k )$ is almost surely bounded and nonnegative.
Therefore it has almost surely a nonnegative limit inferior and a finite limit superior.
Let $\mathcal{S}_n=\{k : n_t^k =n\}$ be the iterations where the sampled scenario passes through node $n$.
For $k \in \mathcal{S}_n$ we have
{\small{
\begin{equation}\label{eqproofconvisddp1}
\begin{array}{lll}
0 \leq \mathcal{Q}_t( x_n^k ) - \mathcal{Q}_t^k ( x_n^k ) & \leq &  \mathcal{Q}_t( x_n^k ) - \mathcal{C}_t^k ( x_n^k ) \\
& \leq & \varepsilon_t^k + \mathcal{Q}_t( x_n^k ) - {\underline{\mathcal{Q}}}_t^k ( x_n^k )\\
& \leq & \displaystyle {\bar \varepsilon} + \sum_{m \in C(n)} p_m  \Big[ \mathfrak{Q}_t(x_n^k , \xi_m ) -  {\underline{\mathfrak{Q}}}_t^{k}(x_n^k , \xi_m )\Big] \\
& \leq & \displaystyle {\bar \varepsilon} + \sum_{m \in C(n)} p_m  \Big[ \mathfrak{Q}_t(x_n^k , \xi_m ) -  {\underline{\mathfrak{Q}}}_t^{k-1}(x_n^k , \xi_m )\Big] \\
& \leq & \displaystyle {\bar \varepsilon} + \delta_t^k + \sum_{m \in C(n)} p_m  \Big[ \mathfrak{Q}_t(x_n^k , \xi_m ) -  \langle c_m , x_m^k \rangle - \mathcal{Q}_{t+1}^{k-1}( x_m^k ) \Big] \\
& \leq & \displaystyle {\bar \varepsilon} + {\bar {\delta}} + \sum_{m \in C(n)} p_m  \Big[ \underbrace{\mathfrak{Q}_t(x_n^k , \xi_m ) -  \langle c_m , x_m^k \rangle - \mathcal{Q}_{t+1}( x_m^k ) }_{\leq 0\mbox{ by definition of }\mathfrak{Q}_t\mbox{ and }x_m^k} + \mathcal{Q}_{t+1}( x_m^k )  -    \mathcal{Q}_{t+1}^{k-1}( x_m^k ) \Big] \\
& \leq & \displaystyle {\bar \varepsilon} + {\bar {\delta}} + \sum_{m \in C(n)} p_m  \Big[ \mathcal{Q}_{t+1}( x_m^k )  -    \mathcal{Q}_{t+1}^{k-1}( x_m^k ) \Big].
\end{array}
\end{equation}
}}

Using the induction hypothesis, we have for every $m \in C(n)$ that
$$
\varlimsup_{k \rightarrow +\infty} \mathcal{Q}_{t+1}( x_{m}^k ) - \mathcal{Q}_{t+1}^k( x_{m}^k ) \leq ({\bar \delta}  +  {\bar{\varepsilon}})(T-t).
$$
In virtue of Lemma \ref{limsuptechlemma}, this implies
\begin{equation}\label{convsupkm1}
\varlimsup_{k \rightarrow +\infty} \mathcal{Q}_{t+1}( x_{m}^k ) - \mathcal{Q}_{t+1}^{k-1}( x_{m}^k ) \leq ({\bar \delta}  +  {\bar{\varepsilon}})(T-t),
\end{equation}
which, plugged into \eqref{eqproofconvisddp1}, gives
\begin{equation}\label{eqconvinsn}
\varlimsup_{k \rightarrow +\infty, k \in \mathcal{S}_n} \mathcal{Q}_t( x_n^k ) - \mathcal{Q}_t^k ( x_n^k ) \leq ({\bar \delta}  +  {\bar{\varepsilon}})(T-t+1).
\end{equation}
Now let us show by contradiction that 
\begin{equation}\label{contrad1}
\varlimsup_{k \rightarrow +\infty} \mathcal{Q}_t( x_n^k ) - \mathcal{Q}_t^k ( x_n^k ) \leq ({\bar \delta}  +  {\bar{\varepsilon}})(T-t+1).
\end{equation}
If \eqref{contrad1} does not hold then there exists $\varepsilon_0>0$ such that
there is an infinite set of iterations $k$ satisfying 
$\mathcal{Q}_t( x_n^k ) - \mathcal{Q}_t^k ( x_n^k ) > ({\bar \delta}  +  {\bar{\varepsilon}})(T-t+1) + \varepsilon_0$
 and by monotonicity, there is also an infinite set of iterations $k$ in the set
 $K=\{k \geq 1 : \mathcal{Q}_t( x_n^k ) - \mathcal{Q}_t^{k-1} ( x_n^k ) > ({\bar \delta}  +  {\bar{\varepsilon}})(T-t+1) + \varepsilon_0\}$.
Let $k_1<k_2<...$ be these iterations: $K=\{k_1,k_2,\ldots,\}$.
Let $y_n^k$ be the random variable which takes the value 1 if $k \in \mathcal{S}_n$ and $0$ otherwise.
Due to Assumption (H3-S), random variables $y_n^{k_1},y_n^{k_2},\ldots,$ are i.i.d. and have the distribution of $y_n^1$. Therefore by the Strong Law of Large Numbers we get 
$$
\frac{1}{N} \displaystyle \sum_{j=1}^N y_n^{k_j}  \xrightarrow{N \rightarrow +\infty} \mathbb{E}[y_n^1]>0 \mbox{ a.s.}
$$
Now let $z_1<z_2<\ldots$ be the iterations in $\mathcal{S}_n$: $\mathcal{S}_n=\{z_1,z_2,\ldots\}$.
Relation \eqref{eqconvinsn} can be written 
$$
\varlimsup_{k \rightarrow +\infty} \mathcal{Q}_t( x_n^{z_k} ) - \mathcal{Q}_t^{z_k} ( x_n^{z_k} ) \leq ({\bar \delta}  +  {\bar{\varepsilon}})(T-t+1),
$$
which, using Lemma \ref{limsuptechlemma}, implies
$$
\varlimsup_{k \rightarrow +\infty} \mathcal{Q}_t( x_n^{z_k} ) - \mathcal{Q}_t^{z_{k-1}} ( x_n^{z_k} ) \leq ({\bar \delta}  +  {\bar{\varepsilon}})(T-t+1).
$$
Using the fact that $z_k \geq z_{k-1}+1$, we deduce that
$$
\begin{array}{lll}
\varlimsup_{k \rightarrow +\infty, k \in \mathcal{S}_n} \mathcal{Q}_t( x_n^k ) - \mathcal{Q}_t^{k-1} ( x_n^k ) & = &
\varlimsup_{k \rightarrow +\infty} \mathcal{Q}_t( x_n^{z_k} ) - \mathcal{Q}_t^{z_{k}-1} ( x_n^{z_k} ) \\
& \leq & 
\varlimsup_{k \rightarrow +\infty} \mathcal{Q}_t( x_n^{z_k} ) - \mathcal{Q}_t^{z_{k-1}} ( x_n^{z_k} ) \leq ({\bar \delta}  +  {\bar{\varepsilon}})(T-t+1).
\end{array}
$$
Therefore, there can only be a finite number of iterations that are both in $K$ and in $\mathcal{S}_n$. This gives
$$
\frac{1}{N} \displaystyle \sum_{j=1}^N y_n^{k_j}  \xrightarrow{N \rightarrow +\infty} 0 \mbox{ a.s.}
$$
We obtain a contradiction and therefore \eqref{contrad1} must hold.

\par (ii) Using \eqref{eqproofconvisddp1}, we obtain for every $t=2,\ldots,T+1$, and every node $n$ of stage $t-1$, that
\begin{equation}
0 \leq   \sum_{m \in C(n)} p_m \Big[ c_m^T x_m^k   + \mathcal{Q}_{t+1}( x_m^k )\Big]  - \mathcal{Q}_t( x_n^k ) 
\leq {\bar \delta}  +  {\bar{\varepsilon}} + \sum_{m \in C(n)} p_m \Big[ \mathcal{Q}_{t+1}( x_m^k ) - \mathcal{Q}_{t+1}^{k-1}( x_m^k ) \Big].
\end{equation}
Therefore
$$
\varliminf_{k \rightarrow +\infty} \sum_{m \in C(n)} p_m \Big[ c_m^T x_m^k   + \mathcal{Q}_{t+1}( x_m^k ) \Big] -  \mathcal{Q}_t( x_n^k )  \geq 0 \\
$$
and using \eqref{convsupkm1} we get
$$
\varlimsup_{k \rightarrow +\infty} \sum_{m \in C(n)} p_m \Big[ c_m^T x_m^k   + \mathcal{Q}_{t+1}( x_m^k ) \Big]  -  \mathcal{Q}_t(x_n^k ) \leq  ({\bar \delta}  +  {\bar{\varepsilon}})(T-t+1).
$$

\par (iii) We have 
\begin{equation}\label{q1iii}
\begin{array}{lll}
\mathcal{Q}_1 ( x_0 ) \geq {\underline{\mathfrak{Q}}}_1^{k-1} ( x_0 , \xi_1 )  & \geq &
c_1^T x_1^k + \mathcal{Q}_2^{k-1}( x_1^k ) - \delta_1^k \\
& \geq & -{\bar{\delta}} + \mathcal{Q}_1 ( x_0 ) + \mathcal{Q}_2^{k-1}( x_1^k ) - \mathcal{Q}_2 ( x_1^k ). 
\end{array}
\end{equation}
Using \eqref{q1iii} and  \eqref{convsupkm1} with $t=1$, we obtain (iii).\\

\par {\textbf{Additional parameters for ISDDP.}} 
For ISDDP, the maximal
number of iterations allowed for Mosek LP solver to solve subproblems along the iterations of 
ISDDP is given in Table \ref{tablenumberiter0}.

\begin{table}
\centering
\begin{tabular}{c}
\begin{tabular}{|c|c|c|c|}
\hline
\begin{tabular}{c}ISDDP\\iteration\end{tabular} &  $[1,20]$ & $[21,50]$ &  $[51,100]$    \\
\hline
\begin{tabular}{l}
LP solver\\maximal\\
number of\\
iterations at $t$
\end{tabular}
& $\left \lceil (0.4+ 0.6 \frac{(t-2)}{T-2})I_{\max} \right \rceil$ &
$\left \lceil (0.45 + 0.55\frac{(t-2)}{T-2})I_{\max} \right \rceil$  & 
$\left \lceil (0.5 + 0.5\frac{(t-2)}{T-2})I_{\max} \right \rceil$ \\
\hline
\end{tabular}
\\
\begin{tabular}{|c|c|c|c|}
\hline
\begin{tabular}{c}ISDDP\\iteration\end{tabular}  & $[101,200]$  &  $[201,300]$ & $[301,400]$  \\
\hline
\begin{tabular}{l}
LP solver\\maximal\\
number of\\
iterations at $t$
\end{tabular}
&
$\left \lceil(0.55+ 0.45\frac{(t-2)}{T-2})I_{\max} \right \rceil$
& $\left \lceil(0.6+ 0.4\frac{(t-2)}{T-2})I_{\max} \right \rceil$ & 
$\left \lceil (0.65+ 0.35\frac{(t-2)}{T-2})I_{\max} \right \rceil$ \\
\hline
\end{tabular}
\\
\begin{tabular}{|c|c|c|c|}
\hline
\begin{tabular}{c}ISDDP\\iteration\end{tabular}  & $[401,500]$  &  $[501,600]$ & $[601,700]$  \\
\hline
\begin{tabular}{l}
LP solver\\maximal\\
number of\\
iterations at $t$
\end{tabular}
&
$\left \lceil(0.7+ 0.3\frac{(t-2)}{T-2})I_{\max} \right \rceil$
& $\left \lceil(0.75+ 0.25\frac{(t-2)}{T-2})I_{\max} \right \rceil$ & 
$\left \lceil (0.8+ 0.2\frac{(t-2)}{T-2})I_{\max} \right \rceil$ \\
\hline
\end{tabular}
\\
\begin{tabular}{|c|c|c|c|}
\hline
\begin{tabular}{c}ISDDP\\iteration\end{tabular}  & $[701,800]$  &  $[801,900]$ & $>900$  \\
\hline
\begin{tabular}{l}
LP solver\\maximal\\
number of\\
iterations at $t$
\end{tabular}
&
$\left \lceil(0.85+ 0.15\frac{(t-2)}{T-2})I_{\max} \right \rceil$
& $\left \lceil(0.9+ 0.1\frac{(t-2)}{T-2})I_{\max} \right \rceil$ & 
$I_{\max}$ \\
\hline
\end{tabular}
\end{tabular}
\caption{Maximal number of iterations for Mosek LP solver for solving backward and forward passes 
subproblems
as a function of stage $t \geq 2$, ISDDP iteration, and the number $I_{\max}$ of iterations 
used to solve subproblems with SDDP with high accuracy.
In this table, $\left \lceil x \right \rceil$ is the smallest integer 
larger than or equal to $x$.}\label{tablenumberiter0}
\end{table}

\end{document}